\documentclass[11pt, a4paper, english]{article}
\pdfoutput=1 
\usepackage{a4wide}
\usepackage{babel} 
\usepackage[utf8]{inputenc} 
\usepackage{dsfont, mathrsfs} 
\usepackage{amsmath, amssymb, amsthm}
\usepackage{enumerate}
\usepackage{textcomp} 
\usepackage{booktabs} 
\usepackage[mathcal]{euscript}
\usepackage{braket} 
\usepackage[T1]{fontenc} 
\usepackage{ifpdf} 
\usepackage{graphicx}
\usepackage{syntonly} 
\usepackage{url} 
\usepackage{algorithm} 
\usepackage{algorithmic} 

\theoremstyle{plain}
\newtheorem{thm}{Theorem}[section]
\newtheorem{lem}[thm]{Lemma}

\theoremstyle{definition}

\newtheorem{rem}[thm]{Remark}

\numberwithin{equation}{section}

\newcommand{\primalProblemTag}{\mathds{P}}

\newcommand{\dualProblemTag}{\mathds{D}}

\newcommand{\regProblemMYTag}{\mathds{D}^{\gamma}}
\newcommand{\discreteProblemMYTag}{\mathds{D}^{\gamma}_{h}}
\newcommand{\discretePwLinearProblemMYTag}{\mathds{D}^{\gamma}_{h,1}}
\newcommand{\primalProblem}{(\mathds{P})}

\newcommand{\dualProblem}{(\mathds{D})}

\newcommand{\regProblemMY}{(\mathds{D}^{\gamma})}
\newcommand{\discreteProblemMY}{(\mathds{D}^{\gamma}_{h})}
\newcommand{\discretePwLinearProblemMY}{(\mathds{D}^{\gamma}_{h,1})}

\newcommand{\slater}{s}

\newcommand{\rst}[1]{\ensuremath{{\mathbin\vert}\raise-.5ex\hbox{$#1$}}}

\newcommand{\dsN}{\ensuremath{\mathds{N}}}
\newcommand{\dsR}{\ensuremath{\mathds{R}}}
\newcommand{\defined}{\ensuremath{\;{\mathrel{\mathop:}=}\;}}
\newcommand{\defines}{\ensuremath{\;{=\mathrel{\mathop:}}\;}}
\newcommand{\ud}{\ensuremath{\,\mathrm{d}}}
\providecommand{\abs}[2][\relax]{%
        \ifx \relax #1 \relax 
                \left \lvert #2 \right \rvert
        \else
                #1 \lvert #2 #1 \rvert
        \fi}
\providecommand{\norm}[2][\relax]{%
        \ifx \relax #1 \relax 
                \left \lVert #2 \right \rVert
        \else
                #1 \lVert #2 #1 \rVert
        \fi}
\providecommand{\innerProd}[3][\relax]{%
        \ifx \relax #1 \relax 
                \left ( #2 \, , \, #3 \right )
        \else
                #1 ( #2 \, , \, #3 #1 )
        \fi}
\providecommand{\dualPair}[3][\relax]{%
        \ifx \relax #1 \relax 
                \left \langle #2 \, , \, #3 \right \rangle
        \else
                #1 \langle #2 \, , \, #3 #1 \rangle
        \fi}
\newcommand{\divergence}{\ensuremath{\operatorname{div}}}
\newcommand{\sgn}{\ensuremath{\operatorname{sgn}}}

\begin{document}
\title{An optimal shape design problem for plates}
\author{Klaus Deckelnick\thanks{Institut f\"ur Analysis und Numerik,
  Otto--von--Guericke--Universit\"at Magdeburg, Universit\"atsplatz 2,
  39106 Magdeburg, Germany (klaus.deckelnick@ovgu.de).}
 \and Michael Hinze\thanks{Schwerpunkt Optimierung und Approximation,
  Universit\"at Hamburg, Bundesstra\ss e 55, 20146 Hamburg, Germany
  (michael.hinze@uni-hamburg.de; tobias.jordan@uni-hamburg.de).}
 \and Tobias Jordan\footnotemark[2]}
\date{}
\maketitle

\begin{abstract}
 We consider an optimal shape design problem for the plate equation, 
 where the variable thickness of the plate is the design function.
 This problem can be formulated as a control in the coefficient PDE-constrained
 optimal control problem with additional control and state constraints.
 The state constraints are treated with a Moreau-Yosida regularization of a
 dual problem.
 Variational discretization is employed for discrete approximation of the
 optimal control problem.
 For discretization of the state in the mixed formulation we compare the
 standard continuous piecewise linear ansatz with a piecewise constant one
 based on the lowest-order Raviart-Thomas mixed finite element.
 We derive bounds for the discretization and regularization errors and also address
 the coupling of the regularization parameter and finite element grid size.
 The numerical solution of the optimal control problem is realized with a
 semismooth Newton algorithm.
 Numerical examples show the performance of the method.
\end{abstract}

\begin{description}
 \item[Key words.] elliptic optimal control problem, optimal shape design, pointwise state
  constraints, Moreau-Yosida regularization, error estimates.
 \item[Mathematics Subject Classification (2010).] 49J20, 65N12, 65N30.
\end{description}

\section{Introduction}
This work is devoted to the numerical analysis and solution of an optimal
control problem for a plate with variable thickness. The state
equation
\begin{align*}
	\Delta(u^3\Delta y) & = f  \qquad \text{in $\Omega$}
\end{align*}
can be used to model the relation between the (small) deflection~$y$ and the
thickness~$u$ of a (thin) plate under the force of a transverse load~$f$. The
domain~$\Omega \subset \dsR^{2}$ represents the unloaded plate's midplane and we
assume its boundary to be simply supported, i.e.,
\begin{align*}
		y = \Delta y = 0  \qquad \text{on $\partial\Omega$}.
\end{align*}
Invoking a pointwise lower bound on the state~$y$ and pointwise almost
everywhere box constraints on the control~$u$, and minimizing the volume of the
plate given by the cost functional
\begin{align*}
	\int_{\Omega}{u(x) \ud x}
\end{align*}
lead to a control in coefficients problem, which can also be viewed as an
optimal shape design problem. In~\cite{SprekelsTiba1999} this optimization
problem is analyzed using a transformation to a dual problem.

Building upon this duality, it is our aim to solve the control
problem with a finite element approximation that is suitable with regard to
the necessary optimality conditions. To this end we compare variational
discretization of the control problem (cf.~\cite{Hinze2005}) based on either
the lowest-order Raviart-Thomas mixed finite element or piecewise linear
continuous finite elements for the discretization of the Poisson equation.
The pointwise state constraints, which
are responsible for the low regularity of the Lagrange multiplier, are treated
with the help of Moreau-Yosida regularization
(cf.~\cite{HintermuellerKunisch2006_1}). The numerical solution to the control
problem is computed via a path-following algorithm that simultaneously refines
the mesh and follows the homotopy generated by the regularization parameter.
The resulting subproblems are solved by a semismooth Newton method.

To the best of the authors' knowledge this is the first contribution to numerical
analysis of a ``control in the coefficients'' problem for biharmonic equations
including state constraints. The mathematical techniques applied in the numerical
analysis of the regularized control problem are related to the relaxation of
state constraints as proposed
in~\cite{HintermuellerHinze2009} and to~\cite{DeckelnickGuentherHinze2009}, where
the Raviart-Thomas mixed finite element was employed in the context of gradient
constraints.

The present work is organized as follows. In Section~\ref{sec:probs} the
optimal control problem and its dual problem are introduced. The regularization
of the dual problem is investigated in Section~\ref{sec:reg}.
Section~\ref{sec:discr} deals with the discretization of the
regularized problems and with the related error bounds. Finally, in
Section~\ref{sec:num} the original control problem is solved with a Newton-type
path-following method. Numerical examples are presented which validate our
analytical findings.
\section{The optimization problems} \label{sec:probs}
Let~$\Omega \subset \dsR^{d}$, $d \in \{2,3\}$, be a bounded domain with smooth
boundary~$\partial \Omega$.
The Dirichlet problem for the Poisson equation
\begin{align} \label{eq:dirichlet_diff}
 \begin{split}
  -\Delta y & = g \quad \text{in $\Omega$,} \\
  y         & = 0 \quad \text{on $\partial \Omega$}
 \end{split}
\end{align}
admits for every~$g \in L^{2}(\Omega)$ a unique solution~$y \defined T(g) \in V
\defined H^{2}(\Omega) \cap H^{1}_{0}(\Omega)$ satisfying
\begin{equation} \label{eq:dirichlet_strong_estimate}
 \norm{y}_{H^{2}(\Omega)} \leq C \norm{g}_{L^{2}(\Omega)}.
\end{equation}

In order to define the control problems considered in this paper we introduce
the admissible sets for controls and states according to
\begin{align*}
 U_{\mathrm{ad}} \defined & \Set{u \in L^{\infty}(\Omega) |
  m \leq u \leq M \quad \text{a.e.\ in $\Omega$} \vphantom{M^{-3}}}, \\
 L_{\mathrm{ad}} \defined & \Set{l \in L^{\infty}(\Omega) |
  M^{-3} \leq l \leq m^{-3} \quad \text{a.e.\ in $\Omega$}}, \\
 Y_{\mathrm{ad}} \defined & \Set{y \in C(\bar{\Omega}) | y \geq - \tau \quad
  \text{in $\Omega$} \vphantom{M^{-3}}},
\end{align*}
where~$\tau > 0$ and~$0< m < M$ are positive real constants. For a given~$f \in
L^{2}(\Omega)$ we consider the following optimal control problems
(cf.~\cite{SprekelsTiba1999}, problems~$\mathbf{P}_{1}$ and~$\mathbf{D}_{1}$):
\begin{align}
 & \min_{u \in L^{\infty}(\Omega)} \tilde{J}(u) \defined \int_{\Omega}{u \ud x}
 \tag{$\primalProblemTag$}
\end{align}
subject to
\begin{align}
 &&&&&& \Delta(u^3\Delta y) & = f        && \text{in $\Omega$}, &&&&&
  \label{def:primal_system_pde} \\
 &&&&&& y                   & = 0        && \text{on $\partial\Omega$},
  \label{def:primal_system_state_bd} \\
 &&&&&& -\Delta y           & = 0        && \text{on $\partial\Omega$},
  \label{def:primal_system_laplacian_bd} \\
 &&&&&& u                   & \in U_{\mathrm{ad}}, &&
  \nonumber \\
 &&&&&& y                   & \in Y_{\mathrm{ad}},
  \nonumber
\end{align}
denoted as the primal problem~$\primalProblem$, representing the physical
control problem motivated in the introduction, and secondly, with the
datum~$z \in V$ induced by~$f$
\begin{align*}
  -\Delta z & = f \quad \text{in $\Omega$,} \\
  y         & = 0 \quad \text{on $\partial \Omega$},
\end{align*}
the dual problem~$\dualProblem$
\begin{align}
 \min_{l \in L^{\infty}(\Omega)} J(l) \defined \int_{\Omega}{l^{-1/3} \ud x}
 \tag{$\dualProblemTag$}
\end{align}
subject to
\begin{align}
 &&&&&& -\Delta y     & = z\,l      && \text{in $\Omega$}, &&&&&
  \label{def:dual_system_state_eq} \\
 &&&&&& y             & = 0         && \text{on $\partial\Omega$},
  \label{def:dual_system_state_bd} \\
 &&&&&& l             & \in L_{\mathrm{ad}}, &&
  \nonumber \\ 
&&&&&& y             & \in Y_{\mathrm{ad}},
  \nonumber
\end{align}
which is analytically and numerically advantageous, in that it is convex
and contains two coupled second order equations instead of the fourth order
equation~\eqref{def:primal_system_pde}. It will therefore serve as a basis for
our analysis in the remaining sections.
For every~$u \in U_{\mathrm{ad}}$ the
system~\eqref{def:primal_system_pde}--\eqref{def:primal_system_laplacian_bd}
has a unique weak solution~$y = y(u) \in V$ with~$u^{3} \Delta y \in V$.
Due to~$z \in L^{2}(\Omega)$
we have~$z\,l \in L^{2}(\Omega)$, and there is a strong
solution~$y = y(l) \in V$ to the dual
system~\eqref{def:dual_system_state_eq}--\eqref{def:dual_system_state_bd}.

We impose the following Slater condition:
\begin{align} \label{assum:slater}
 \exists u_{\slater} \in U_{\mathrm{ad}}, \varepsilon_{\slater} > 0 \colon \quad
 y_{\slater} = y(u_{\slater}) > -\tau + \varepsilon_{\slater},
\end{align}
and recall from~\cite{SprekelsTiba1999} that for each pair~$(y,u)$ admissible
for~$\primalProblem$, the pair~$(y,l = u^{-3})$ is admissible
for~$\dualProblem$ with the same cost and vice versa. Moreover,
there exists a unique solution~$u_{\text{opt}}$ of problem~$\primalProblem$,
and $l_{\text{opt}} = u_{\text{opt}}^{-3}$ is the unique solution
of~$\dualProblem$. We denote the associated state by~$y_{\text{opt}}$.

Next we derive optimality conditions characterizing~$l_{\text{opt}}$.
For this purpose we introduce~$\mathscr{M}(\bar{\Omega})$, the space of
regular Borel measures, which equipped with norm
\begin{align*}
 \norm{\cdot}_{\mathscr{M}(\bar{\Omega})} = \sup_{f \in C(\bar{\Omega}),
 \abs{f} \leq 1} \int_{\bar{\Omega}} f \ud (\cdot)
\end{align*}
is the dual space of~$C(\bar{\Omega})$.
Arguments similar to those used in~\cite[Theorem 2]{Casas1986} yield the
\begin{thm} \label{thm:optimality_conditions}
 Let the assumption~\eqref{assum:slater} hold.
 A control~$l \in L^{\infty}(\Omega)$
 with associated state~$y = y(l)$ is optimal for the dual
 problem~$\dualProblem$ if and only if there exist~$q \in L^2(\Omega)$
 and~$\nu \in \mathscr{M}(\bar{\Omega})$, such that
 \begin{align}
  &&&&&& -\int_{\Omega} q\, \Delta w \ud x &= \int_{\Omega} w \ud \nu(x)
   && \forall \, w \in V, \nonumber &&&&& \\
  &&&&&& \int_{\Omega} \left( q\, z - \frac{1}{3}\, l^{-4/3} \right) (k - l) \ud x
   & \geq 0 && \forall \, k \in L_{\mathrm{ad}}, \label{os:var_ineq} &&&&& \\
  &&&&&& \int_{\bar{\Omega}} (w - y) \ud \nu(x)	& \leq 0 && \forall \, w \in Y_{\mathrm{ad}},
   \label{os:compl_ineq} &&&&& \\
  &&&&&& l \in L_{\mathrm{ad}},& \qquad y \in Y_{\mathrm{ad}}
  \nonumber
 \end{align}
 are satisfied. 
\end{thm}
The variational inequality~\eqref{os:var_ineq} can be written as a
projection formula
\begin{equation} \label{os:proj}
 l = \left( P_{\left[m^{4},M^{4}\right]} \left( 3\,q\,z \right)
  \right)^{-3/4} \quad \mbox{a.e.\ in $\Omega$},
\end{equation}
where~$P_{[a,b]}$ is the orthogonal projection onto the real interval~${[a,b]}$.
For later use, we note that~\eqref{os:compl_ineq} is equivalent to
\begin{equation} \label{os:compl}
 \nu \leq 0, \quad	\int_{\bar{\Omega}} (y + \tau) \ud \nu(x) = 0.
\end{equation}
It follows from~\eqref{os:compl} that the support of the measure~$\nu$ is
concentrated in the state-active set~$\Set{ x \in \Omega | y(x) = -\tau}$.
In particular~$\nu_{|\partial \Omega} \equiv 0$, see~\cite{Casas1986}. Furthermore,
from Theorems~4 and~5 in~\cite{Casas1986} we deduce~$q \in W^{1,s}_{0}(\Omega)$
for all~$1 \leq s < d/(d-1)$ and~$q \in H^{2}_{\text{loc}}(\Omega \setminus
\{y = -\tau\})$.
\section{Moreau-Yosida regularized problem} \label{sec:reg}
To relax the state constraints, we introduce the Moreau-Yosida
regularization~$\regProblemMY$ of
problem~$\dualProblem$ for a parameter~$\gamma > 0$. It reads
\begin{align}
  \min_{l \in L_{\text{ad}}} J^{\gamma}(l)
  \defined \int_{\Omega}{l^{-1/3} \ud x} + \frac{\gamma}{2} \int_{\Omega}
  \big( (y + \tau)^{-} \big)^{2} \ud x
  \tag{$\regProblemMYTag$}
\end{align}
subject to
\begin{align*}
 y & = T(z\,l).
\end{align*}
Here we set~$(\cdot)^{-} \defined \min\{0 , (\cdot)\}$.
It admits a unique solution~$l^{\gamma}$ of $\regProblemMY$
with associated state denoted by~$y^{\gamma} = T(z \, l^{\gamma})$. Furthermore,
there exists a unique~$q^{\gamma} \in L^2(\Omega)$ which together
with~$y^{\gamma}$ and~$l^{\gamma}$ satisfies
\begin{align}
 &&&& -\int_{\Omega} q^{\gamma} \, \Delta w \ud x  & =
  \int_{\Omega} \gamma (y^{\gamma} + \tau)^{-} \, w \ud x
  && \forall \, w \in V, &&&&&
  \label{os_my:adj_eq_weak} \\
 &&&& \int_{\Omega} \left( q^{\gamma}\, z - \frac{1}{3}\, (l^{\gamma})^{-4/3} \right)
  (k - l^{\gamma}) \ud x & \geq 0 && \forall \, k \in L_{\text{ad}},
  \label{os_my:var_ineq} \\
 &&&& l^{\gamma}                 & \in L_{\text{ad}}.
  \nonumber
\end{align}
The term~$\nu^{\gamma} \defined \gamma (y^{\gamma} + \tau)^{-}$ can be regarded
as a regularized version of the Lagrange multiplier~$\nu \in
\mathscr{M}(\bar{\Omega})$ in Theorem~\ref{thm:optimality_conditions}.
There holds~$q^{\gamma} = T(\nu^{\gamma}) \in V$ and
\begin{align*}
 \dualPair{y^{\gamma} - w}{\nu^{\gamma}}_{C(\bar{\Omega}),
 \mathscr{M}(\bar{\Omega})} \geq 0 \qquad \forall \, w \in Y_{\text{ad}}.
\end{align*}
This inequality for~$w \in Y_{\text{ad}}$ can be argued as follows:
\begin{align*}
 \dualPair{y^{\gamma} - w}{\nu^{\gamma}}_{C(\bar{\Omega}),
 M(\bar{\Omega})}
 & = \innerProd{y^{\gamma} - w}{\nu^{\gamma}}_{L^{2}(\Omega)}
 = \innerProd{y^{\gamma} + \tau - \tau - w}{\nu^{\gamma}}_{L^{2}(\Omega)} \\
 & = \innerProd{y^{\gamma} + \tau}{\gamma(y^{\gamma} + \tau)^{-}}_{L^{2}(\Omega)} +
 \innerProd{-(w + \tau)}{\nu^{\gamma}}_{L^{2}(\Omega)} \\
 & = \gamma \norm{(y^{\gamma} + \tau)^{-}}_{L^{2}(\Omega)}^{2} + 
 \innerProd{-(w + \tau)}{\nu^{\gamma}}_{L^{2}(\Omega)} \geq 0.
\end{align*}

Now we want to show convergence of the parameterized subproblems~$\regProblemMY$
towards the unregularized problem~$\dualProblem$. We begin with
uniform boundedness of primal and dual variables with respect to~$\gamma$.
While the former is obtained immediately through the control constraints
and~\eqref{eq:dirichlet_strong_estimate}, the latter can be shown as follows.
\begin{lem} \label{lem:my_bounds}
 Let~$\gamma > 0$ and~$l^{\gamma} \in L_{\mathrm{ad}}$ be the solution to
 the problem~$\regProblemMY$ with associated state~$y^{\gamma}$ and
 multipliers~$q^{\gamma}, \nu^{\gamma}$ according to the optimality conditions.
 Then there exists a constant~$C > 0$ independent of~$\gamma$ such that
 \begin{align*}
  \norm{\nu^{\gamma}}_{L^{1}(\Omega)}, \quad
  \norm{q^{\gamma}}_{L^{2}(\Omega)} \quad \leq \quad C.
 \end{align*}
\end{lem}
\begin{proof}
To uniformly bound~$\nu^{\gamma}$ in~$L^{1}(\Omega)$ we 
test~\eqref{os_my:var_ineq} with the Slater element~$l_{\slater}$. With the
help of the adjoint equation~\eqref{os_my:adj_eq_weak} we get
 \begin{align*}
  C \geq \innerProd{\frac{1}{3}(l^{\gamma})^{-4/3}}{l^{\gamma} -
  l_{\slater}}_{L^{2}(\Omega)} \geq \innerProd{q^{\gamma}}{z
  (l^{\gamma} - l_{\slater}}_{L^{2}(\Omega)}
  = \innerProd{\nu^{\gamma}}{y^{\gamma} - y_{\slater}}_{L^{2}(\Omega)},
 \end{align*}
 with a constant~$C$ independent of~$\gamma$. The desired estimate now
 follows from
 \begin{align*}
  \innerProd{\nu^{\gamma}}{y^{\gamma} - y_{\slater}}_{L^{2}(\Omega)}
  & = \innerProd{\nu^{\gamma}}{y^{\gamma} + \tau - \tau
  - y_{\slater}}_{L^{2}(\Omega)}\\
  & = \gamma \norm{(y^{\gamma} + \tau)^{-}}_{L^{2}(\Omega)}^{2}
  + \innerProd{-\nu^{\gamma}}{y_{\slater} + \tau}_{L^{2}(\Omega)} \\
  & \geq 0 + \innerProd{-\nu^{\gamma}}{\varepsilon_{\slater}}_{L^{2}(\Omega)} \\
  & = \varepsilon_{\slater} \norm{\nu^{\gamma}}_{L^{1}(\Omega)}.
 \end{align*}
With this bound we want to prove the one on the dual state~$q^{\gamma}$.
Let~$w \in V$ solve
\begin{align*}
  -\Delta w & = q^{\gamma} \quad \text{in $\Omega$,} \\
  w         & = 0          \quad \text{on $\partial \Omega$},
\end{align*}
Using~\eqref{os_my:adj_eq_weak}, the embedding of~$C(\bar{\Omega})$
into~$H^{2}(\Omega)$ and the continuous dependence of~$w$ on~$q^{\gamma}$,
we have that
\begin{align*}
 \norm{q^{\gamma}}_{L^2(\Omega)}^2 &= \int_{\Omega} q^{\gamma} (-\Delta w) \ud x
 = \int_{\Omega} \nu^{\gamma} w \ud x \leq C \norm{\nu^{\gamma}}_{L^1(\Omega)}
 \norm{w}_{L^{\infty}(\Omega)} \\
 &\leq C \norm{\nu^{\gamma}}_{L^1(\Omega)} \norm{w}_{H^{2}(\Omega)}
 \leq C \norm{q^{\gamma}}_{L^2(\Omega)},
\end{align*}
hence~$\norm{q^{\gamma}}_{L^2(\Omega)} \leq C$.
\end{proof}
Next we need to estimate the
violation of the state constraint measured in the maximum norm with the help of
techniques developed in~\cite{HintermuellerSchielaWollner2014}.
\begin{lem} \label{lem:violation_state_constr_order}
  Let~$l^{\gamma}$ be the solution of problem~$\regProblemMY$, $y^{\gamma}$
  the corresponding state. Then for~$d \in \{2,3\}$ we have for
  every~$\varepsilon > 0$ a constant~$C_{\varepsilon} > 0$, independent of~$\gamma$, such that
  \begin{align*}
    \norm{(y^{\gamma} + \tau)^{-}}_{L^{\infty}(\Omega)}
    \leq C_{\varepsilon} \left(\gamma^{d/4 - 1 + \varepsilon} \right).
  \end{align*}
\end{lem}
\begin{proof}
 We show that Corollary~2.6 of~\cite{HintermuellerSchielaWollner2014} is
 applicable. To begin with, we note that~$l^{\gamma} = P(3 q^{\gamma} z)^{-3/4}$
 is uniformly bounded in~$W^{1,s}$ for every~$s \in [1,d/(d-1))$,
 since~$\nu^{\gamma}$ is uniformly bounded in~$L^{1}(\Omega)$. This implies
 that~$zl^{\gamma} \in W^{1,s}(\Omega)$ is uniformly bounded in~$\gamma$
 and that~$y^{\gamma} \in W^{3,s}(\Omega)$ is uniformly bounded
 w.r.t.~$\gamma$, which by Sobolev imbedding theorems holds also in~$C^{\beta}
 (\bar{\Omega})$, for~$\beta = 4 - d - \varepsilon$ and all~$\varepsilon > 0$.
 Thus Corollary~2.6 in~\cite{HintermuellerSchielaWollner2014}
 is applicable and delivers our desired bound.
\end{proof}
 We are now in position to estimate the regularization error.
\begin{thm} \label{thm:error_l_gamma}
 Let~$l$ and~$l^{\gamma}$ be the solutions to~$\dualProblem$
 and~$\regProblemMY$, resp., with corresponding states~$y$ and~$y^{\gamma}$.
 Then for every~$\varepsilon > 0$ there exists a constant~$C_{\varepsilon}>0$,
 independent of~$\gamma$, for which it holds that
 \begin{align*}
  \norm{l - l^{\gamma}}_{L^{2}(\Omega)} +
  \norm{y - y^{\gamma}}_{H^{2}(\Omega)} +
  \norm{y - y^{\gamma}}_{L^{\infty}(\Omega)} \leq
  C_{\varepsilon} \gamma^{-\frac{1}{2} \left(1 - \frac{d}{4}\right) + \varepsilon}.
 \end{align*}
\end{thm}
\begin{proof}
 Using~$l$ as test function in~\eqref{os_my:var_ineq} and~$l^{\gamma}$
 as test function in~\eqref{os:var_ineq} we obtain
 \begin{align*}
  & \mathrel{\phantom{=}} C \norm{l - l^{\gamma}}_{L^{2}(\Omega)}^{2} \\
  & \leq \innerProd{q^{\gamma} - q}{z\,l - z\,l^{\gamma}}_{L^{2}(\Omega)}
  = -\innerProd{q^{\gamma} - q}{\Delta (y - y^{\gamma})}_{L^{2}(\Omega)} \\
  & = \innerProd{y - y^{\gamma}}{\gamma(y^{\gamma} + \tau)^{-}}_{L^{2}(\Omega)} -
  \dualPair{y - y^{\gamma}}{\nu}_{C(\bar{\Omega}),\mathscr{M}(\bar{\Omega})} \\
  & = \innerProd{y + \tau}{\gamma(y^{\gamma} + \tau)^{-}}_{L^{2}(\Omega)} -
  \innerProd{y^{\gamma} + \tau}{\gamma(y^{\gamma} + \tau)^{-}}_{L^{2}(\Omega)} +
  \dualPair{y^{\gamma} - y}{\nu}_{C(\bar{\Omega}),\mathscr{M}(\bar{\Omega})} \\
  & \defines (I) + (II) + (III).
 \end{align*}
 Since~$y \in Y_{\mathrm{ad}}$ we have~$(I) \leq 0$. Moreover~$(II) = -\gamma
 \norm{(y^{\gamma} + \tau)^{-}}_{L^{2}(\Omega)}^{2} \leq 0$. The third addend is
 treated with the complementarity condition~\eqref{os:compl} for
 the multiplier~$\nu$:
 \begin{align*}
  (III) & = \dualPair{y^{\gamma} + \tau}{\nu}_{C(\bar{\Omega}),
  \mathscr{M}(\bar{\Omega})} + \underbrace{\dualPair{-y - \tau}{\nu}_{C(\bar{\Omega}),
  \mathscr{M}(\bar{\Omega})}}_{= 0} \\[-1.5em]
  & = \dualPair{y^{\gamma} + \tau}{\nu}_{C(\bar{\Omega}),
  \mathscr{M}(\bar{\Omega})} \\
  & \leq \dualPair{(y^{\gamma} + \tau)^{-}}{\nu}_{C(\bar{\Omega}),
  \mathscr{M}(\bar{\Omega})}.
 \end{align*}
 With the help of Lemma~\ref{lem:violation_state_constr_order} we arrive at
 \begin{align*}
  C \norm{l - l^{\gamma}}_{L^{2}(\Omega)}^{2} & \leq \dualPair{(y^{\gamma} +
  \tau)^{-}}{\nu}_{C(\bar{\Omega}),\mathscr{M}(\bar{\Omega})} \leq
  \norm{(y^{\gamma} + \tau)^{-}}_{L^{\infty}(\Omega)} \norm{\nu}_{\mathscr{M}
  (\bar{\Omega})} \\
  & \leq C_{\varepsilon} \left( \gamma^{d/4 - 1 + \varepsilon} \right),
 \end{align*}
 with $\gamma$-independent constants~$C$, $C_{\epsilon}$. The continuous dependence
 of the states on the controls and the continuous embedding~$H^{2}(\Omega)
 \hookrightarrow C(\bar{\Omega})$ allow to extend this estimate to~$\norm{y -
 y^{\gamma}}_{H^{2}(\Omega)}$ and~$\norm{y - y^{\gamma}}_{L^{\infty}(\Omega)}$.
\end{proof}
\section{Finite element discretization} \label{sec:discr}
%
\subsection{Mixed piecewise constant versus piecewise linear approximation}
\label{subsec:pwconstant_vs_pwlinear}
In this section we turn to the variational discretization of the regularized
control problems, taking into account the structure imposed by the optimality
systems, especially the projection formula~\eqref{os:proj} and its discrete
counterparts~\eqref{os_my_h:proj} and~\eqref{os_my_h:proj_p1}. The
function~$(P_{\left[m^{4},M^{4}\right]}(\cdot))^{-3/4}$ applied
to the product of two state variables is evaluated with little effort if those
variables are approximated piecewise constant and yields an implicit piecewise
constant discretization of the optimal control. Further, the finite element
system of the semismooth Newton method in Section~\ref{sec:num},
in particular the parts~\eqref{eq:newton_k} and~\eqref{eq:newton_Dk}
involving the projection formula and its generalized derivative, in this
situation is easily assembled exactly.

Approximating the states with piecewise linear, continuous finite elements
delivers a more involved variational discretization of the controls, since
the projection formula then no longer implies a piecewise polynomial
discretization of the control variable, but rather the negative power of the
pointwise projection of a piecewise quadratic function. Moreover, the
approximate computation of the terms~\eqref{eq:newton_k_p1}
and~\eqref{eq:newton_Dk_p1} introduces an additional error. On the other hand,
a piecewise linear ansatz delivers the higher approximation order two for the
states, as opposed to an order of at most one for a piecewise constant ansatz.
This is supported by the convergence rates w.r.t.\ the grid size~$h$
in the error plots in Section~\ref{sec:num}, and also allows for a better
resolution of the control active sets.

In the remainder of this section we give estimates for the overall
error in both discrete approaches.
\subsection{Variational discretization of~$\regProblemMY$ with mixed finite elements}
Following the above remarks, we use a mixed finite element method based
on the lowest-order Raviart-Thomas element. To begin with we
recall the mixed formulation of the Dirichlet-problem for the
Poisson equation, i.e., for~$g \in L^{2}(\Omega)$, $y = T(g)$
and~$\mathbf{v} = \nabla y$ there holds
\begin{align}
 \begin{split} \label{eq:mixed_poisson}
	\int_{\Omega} \mathbf{v} \cdot \mathbf{w} \ud x +
	\int_{\Omega} y \, \divergence \mathbf{w} \ud x
	&= 0 \qquad \forall \, \mathbf{w} \in H(\divergence,\Omega), \\
	\int_{\Omega} \phi \, \divergence \mathbf{v} \ud x +
	\int_{\Omega} g \, \phi \ud x
	&= 0 \qquad \forall \, \phi \in L^{2}(\Omega),
 \end{split}
\end{align}
where~$H(\divergence,\Omega) \defined
\Set{\mathbf{w} \in L^{2}(\Omega)^{d} | \divergence \mathbf{w} \in L^{2}(\Omega) }$.
For a given right-hand side~$g \in L^{2}(\Omega)$ we represent the solution
of this mixed problem by~$G(g) \defined (y , \mathbf{v})$. In particular, with~$\mathbf{v}_z \defined \nabla z$ this means $G(f) = (z , \mathbf{v}_z)$.

Let a triangulation~$\mathcal{T}_{h}$ of~$\Omega$ be given, where~$h \defined
\max_{T \in \mathcal{T}_{h}} \operatorname{diam}(T)$ and~$\bar{\Omega}$ be
the union of the elements of~$\mathcal{T}_{h}$, with boundary elements allowed
to have one curved face. We additionally assume that the triangulation is
quasi-uniform, i.e., there exists a constant~$\rho > 0$, independent of~$h$,
such that each~$T \in \mathcal{T}_{h}$ is contained in a ball of
radius~$\rho^{-1} h$ and contains a ball of radius~$\rho h$. To define the
discrete version of~\eqref{eq:mixed_poisson} let us introduce the spaces
\begin{align*}
	\mathbf{V}_{h} & \defined RT_{0}(\Omega,\mathcal{T}_{h}) \defined \Set{
	\mathbf{w}_{h} \in
	H(\divergence,\Omega) | \mathbf{w}_{h|T} \in RT_{0}(T) \quad
	\forall \, T \in \mathcal{T}_{h}}, \\
	RT_{0}(T) & \defined \Set{\mathbf{w} \colon T \to \dsR^{d} |
	\exists \, a \in \dsR^{d}, \exists \, \beta \in \dsR :
	\mathbf{w}(x) = a + \beta x \quad \forall \, x \in \dsR^{d}}, \\
	Y_{h} & \defined \Set{\phi_{h} \in L^{2}(\Omega) |\forall \, T \in \mathcal{T}_{h} \ 
	\exists \, \beta_{T} \in \dsR : \phi_{h|T} \equiv \beta_{T}}.
\end{align*}
For a given $g \in L^{2}(\Omega)$ we set~$G_{h}(g) \defined (y_{h} ,
\mathbf{v}_{h}) \in Y_{h} \times \mathbf{V}_{h}$ to be the solution of
\begin{align*}
	\int_{\Omega} \mathbf{v}_{h} \cdot \mathbf{w}_{h} \ud x +
	\int_{\Omega} y_{h} \, \divergence \mathbf{w}_{h} \ud x
	&= 0 \qquad \forall \, \mathbf{w}_{h} \in \mathbf{V}_{h}, \\
	\int_{\Omega} \phi_{h} \, \divergence \mathbf{v}_{h} \ud x +
	\int_{\Omega} g \, \phi_{h} \ud x
	&= 0 \qquad \forall \, \phi_{h} \in Y_{h}.
\end{align*}
The resulting error satisfies~(see~\cite{BrezziFortin1991})
\begin{align} \label{eq:discretization_error_L2}
\begin{split}
	\norm{y - y_{h}}_{L^{2}(\Omega)} +
	 \norm{\mathbf{v} - \mathbf{v}_{h}}_{L^{2}(\Omega)^{d}}
	&\leq Ch \left( \norm{y}_{H^{1}(\Omega)} +
	 \norm{\mathbf{v}}_{H^{1}(\Omega)^{d}} \right) \\
	&\leq Ch \norm{y}_{H^{2}(\Omega)} \leq Ch \norm{g}_{L^{2}(\Omega)},
\end{split}
\end{align}
as well as, if~$g \in L^{\infty}(\Omega)$, the pointwise
estimate~(see~\cite[Cor.~5.5]{GastaldiNochetto1989})
\begin{align} \label{eq:discretization_error_LInfty}
	\norm{y - y_{h}}_{L^{\infty}(\Omega)} +
	\norm{\mathbf{v} - \mathbf{v}_{h}}_{L^{\infty}(\Omega)^{d}}
	\leq C h \abs{\log h} \norm{g}_{L^{\infty}(\Omega)}.
\end{align}
The load~$f$ induces a discrete datum~$z_{h} \in Y_{h}$ via~$(z_{h},
\mathbf{v}_{z,h}) = G_{h}(f)$.

\begin{rem}
 For our error analysis we require that~$\norm{z_{h}}_{L^{\infty}(\Omega)}$
 is bounded uniformly in~$h$ in both of the considered discretization
 approaches.
 
 This is satisfied, e.g., if~$f \in L^{\infty}(\Omega)$, since
 then~$\norm{z - z_{h}}_{L^{\infty}(\Omega)} \to 0$
 by~\eqref{eq:discretization_error_LInfty}
 and~\eqref{eq:discretization_error_LInfty_linear}, resp.
 Note that this regularity restriction is not essentially necessary,
 cf.~\cite[Lemma 3.4]{DeckelnickHinze2014}, which holds analogously for the
 scalar states in both discrete approaches and allows for~$f \in
 L^{p}(\Omega)$, $p>2$, with smaller powers of~$h$
 in~\eqref{eq:discretization_error_LInfty}
 and~\eqref{eq:discretization_error_LInfty_linear}.
\end{rem}

Let~$Y_{\mathrm{ad},h} \defined \Set{\phi_{h} \in Y_{h} | \phi_{h|T} \geq -\tau
\quad \forall T \in \mathcal{T}_{h}}$.
The variational discretization~$\discreteProblemMY$ of the regularized control
problems~$\regProblemMY$ reads
\begin{align}
 \min_{l \in L^{\infty}(\Omega)} J^{\gamma}_{h}(l)
  \defined \int_{\Omega}{l^{-1/3} \ud x} + \frac{\gamma}{2} \int_{\Omega}
  \big( (y_{h} + \tau)^{-} \big)^{2} \ud x
 \tag{$\discreteProblemMYTag$}
\end{align}
subject to
\begin{align*}
 &&&&&& (y_{h} , \mathbf{v}_{h})  & = G_{h}(z_{h} \, l),      &&&&&
  \\
 &&&&&& l                 & \in L_{\mathrm{ad}}.
\end{align*}
We note that~$\discreteProblemMY$ is still an
infinite-dimensional optimization problem similar to~$\regProblemMY$ since the
control~$l$ is not discretized. It admits a unique solution~$l^{\gamma}_{h}$,
which is characterized by the optimality system
\begin{align}
 &&&& (y^{\gamma}_{h} , \mathbf{v}^{\gamma}_{h})   & =
  G_{h}(z_{h} \, l^{\gamma}_{h}),      &&&&&
  \label{os_my_h:prim_eq} \\
 &&&& (q^{\gamma}_{h} , \mathbf{v}^{\gamma}_{q,h}) & =
  G_{h} \left( \gamma(y^{\gamma}_{h} + \tau)^{-} \right),
  \label{os_my_h:adj_eq} \\
  &&&& \int_{\Omega} \left( q^{\gamma}_{h}\, z_{h} - \frac{1}{3} \,
   (l^{\gamma}_{h})^{-4/3} \right) (k - l^{\gamma}_{h}) \ud x
   & \geq 0 \qquad \forall \, k \in L_{\mathrm{ad}},
   \label{os_my_h:var_ineq} \\
  &&&& l^{\gamma}_{h}                 & \in L_{\mathrm{ad}}.
   \nonumber
\end{align}
Condition~\eqref{os_my_h:var_ineq} is equivalent to the projection formula
\begin{align} \label{os_my_h:proj}
 l^\gamma_{h} = \left( P_{\left[m^{4},M^{4}\right]} \left( 3\,q^\gamma_{h}\,z_{h}
 \right) \right)^{-3/4}.
\end{align}
We denote~$\nu^{\gamma}_{h} \defined \gamma(y^{\gamma}_{h} + \tau)^{-}$ and
similarly to the proof of Lemma~\ref{lem:my_bounds} one obtains boundedness
uniformly in~$h$ and~$\gamma$:
\begin{lem} \label{lem:nu_gamma_h}
        Let~$\gamma > 0$ and~$l^{\gamma}_{h} \in L_{\mathrm{ad}}$ be the solution to
        the problem~$\discreteProblemMY$ with state~$(y^{\gamma}_{h} ,
        \mathbf{v}^{\gamma}_{h}) = G_{h}(z_{h} \, l^{\gamma}_{h})$.     
        Then there exists an $h_{0} > 0$ and a constant~$C > 0$ independent
        of~$\gamma$ and of~$h$ such that
        \begin{align*}
                \norm{\nu^{\gamma}_{h}}_{L^{1}(\Omega)} \leq C
                \qquad \forall \, 0 < h < h_{0}, \quad \forall \, \gamma > 0.
        \end{align*}
\end{lem}
\begin{proof}
 Let~$(q^{\gamma}_{h} , \mathbf{v}^{\gamma}_{q,h})$ be the adjoint state
 and~$l_{\slater}$ the Slater element with corresponding discrete
 state~$ (y_{\slater,h} , \mathbf{v}_{\slater,h}) = G_h(z_h l_s)$,
 which is a discrete Slater state for all~$0 < h < h_{0}$ with
 some~$h_{0} > 0$ small enough. In fact, with~$y_{\slater}$
and~$\varepsilon_{\slater}$ from~\eqref{assum:slater} we obtain
from~$\norm{y_{\slater} - y_{\slater,h}}_{L^{\infty}(\Omega)} \to 0$
that~$y_{\slater,h} > -\tau + \varepsilon_{\slater}/2$ 
for~$0 < h < h_{0}$.
 
 We test~\eqref{os_my_h:var_ineq}
 with~$l_{\slater}$ and with the help of the adjoint
 equation~\eqref{os_my_h:adj_eq} and the definition of~$G_{h}$ we get
 \begin{align*}
  C \geq \innerProd{\frac{1}{3}(l^{\gamma}_{h})^{-4/3}}{l^{\gamma}_{h} -
  l_{\slater}}_{L^{2}(\Omega)} \geq \innerProd{q^{\gamma}_{h}}{z_{h}
  (l^{\gamma}_{h} - l_{\slater})}_{L^{2}(\Omega)}
  = \innerProd{\nu^{\gamma}_{h}}{y^{\gamma}_{h} - y_{\slater,h}}_{L^{2}(\Omega)},
 \end{align*}
 with a constant~$C$ independent of~$\gamma$ and~$h$.
 We then have
 \begin{align*}
  \innerProd{\nu^{\gamma}_{h}}{y^{\gamma}_{h} - y_{\slater,h}}_{L^{2}(\Omega)}
  & = \innerProd{\nu^{\gamma}_{h}}{y^{\gamma}_{h} + \tau - \tau
  - y_{\slater,h}}_{L^{2}(\Omega)}\\
  & = \gamma \norm{(y^{\gamma}_{h} + \tau)^{-}}_{L^{2}(\Omega)}^{2}
  + \innerProd{-\nu^{\gamma}_{h}}{y_{\slater,h} + \tau}_{L^{2}(\Omega)} \\
  & \geq 0 + \innerProd{-\nu^{\gamma}_{h}}{\frac{\varepsilon_{\slater}}{2}}_{L^{2}(\Omega)}
  \qquad \forall \, 0 < h < h_{0} \\
  & = \frac{\varepsilon_{\slater}}{2} \norm{\nu^{\gamma}_{h}}_{L^{1}(\Omega)}
  \qquad \forall \, 0 < h < h_{0}.
 \end{align*}
\end{proof}

Aiming for an estimate of the overall
error induced by regularization and discretization we
apply the approach
from~\cite{HintermuellerHinze2009} to our problem setting and derive a
similar asymptotic $h$,$\gamma$-dependent bound on~$(l -
l^{\gamma}_{h})$ in~$L^{2}(\Omega)$, which further allows to couple the
regularization parameter efficiently to the grid size parameter. To this end we
need to estimate the discretization error for the regularized problems.
\begin{thm} \label{thm:error_l_gamma_h}
 Let~$l^{\gamma}$ and~$l^{\gamma}_{h}$ be the solutions of~$\regProblemMY$
 and~$\discreteProblemMY$, resp., with corresponding states~$y^{\gamma}$
 and~$(y^{\gamma}_{h} , \mathbf{v}^{\gamma}_{h})$. Then there is an~$h_{0} > 0$ and
 a~$\gamma$- and~$h$-independent constant~$C$, such that for
 all~$h \in (0,h_{0})$ and all~$\gamma > 0$
 \begin{align*}
  \norm{l^{\gamma} - l^{\gamma}_{h}}_{L^{2}(\Omega)} +
  \norm{y^{\gamma} - y^{\gamma}_{h}}_{L^{\infty}(\Omega)} \leq
  C h^{1/2} \abs{\log h}^{1/2}.
 \end{align*}
\end{thm}
\begin{proof}
 We define the auxiliary variable~$(q^{\nu}_{h},\mathbf{v}^{\nu}_{h}) =
 G_{h}(\nu^{\gamma})$
 and test the problems' variational inequalities with the respective
 solutions to obtain
 \begin{align*}
  &\phantom{\ =\ } C \norm{l^{\gamma} - l^{\gamma}_{h}}_{L^{2}(\Omega)}^{2}
  \leq \innerProd{q^{\gamma}z -
  q^{\gamma}_{h} z_{h}}{l^{\gamma}_{h} - l^{\gamma}}_{L^{2}(\Omega)} \\
  & = \innerProd{q^{\gamma}z - q^{\gamma} z_{h}}
  {l^{\gamma}_{h} - l^{\gamma}}_{L^{2}(\Omega)} +
  \innerProd{q^{\gamma}z_{h} - q^{\nu}_{h} \, z_{h}}
  {l^{\gamma}_{h} - l^{\gamma}}_{L^{2}(\Omega)} +
  \innerProd{q^{\nu}_{h} z_{h} - q^{\gamma}_{h} \, z_{h}}
  {l^{\gamma}_{h} - l^{\gamma}}_{L^{2}(\Omega)} \\
  & \defines (I) + (II) + (III).
 \end{align*}
 In view of Lemma~\ref{lem:my_bounds} we have for sufficiently
 small~$0 < h < h_{0}$ that
 \begin{align*} 
  (I) & \leq \norm{q^{\gamma}}_{L^{2}(\Omega)} \norm{z - z_{h}}_{L^{2}(\Omega)}
   \norm{l^{\gamma}_{h} - l^{\gamma}}_{L^{\infty}(\Omega)} \leq C h
   \norm{l^{\gamma}_{h} - l^{\gamma}}_{L^{\infty}(\Omega)} \leq C h.
 \end{align*}
 For the second addend we find
 \begin{align*} 
  (II) & \leq \norm{z_{h}}_{L^{\infty}(\Omega)}
   \norm{q^{\gamma} - q^{\nu}_{h}}_{L^{1}(\Omega)}
   \norm{l^{\gamma}_{h} - l^{\gamma}}_{L^{\infty}(\Omega)}
   \leq C \norm{q^{\gamma} - q^{\nu}_{h}}_{L^{1}(\Omega)}.
 \end{align*}
 To estimate the finite element error~$\norm{q^{\gamma} -
 q^{\nu}_{h}}_{L^{1}(\Omega)}$ we set~$p = \sgn(q^{\gamma} - q^{\nu}_{h}) \in L^{2}(\Omega)$
 (cf. the proof of Lemma~8.3.11 on p.~228 in~\cite{BrennerScott2008}) and consider
 \begin{align} \label{eq:q_gamma_times_p}
  \innerProd{q^{\gamma} - q^{\nu}_{h}}{p}_{L^{2}(\Omega)} =
   \innerProd{q^{\gamma}}{p}_{L^{2}(\Omega)} -
    \innerProd{q^{\nu}_{h}}{p}_{L^{2}(\Omega)}.
 \end{align}
 Defining~$(y^{p}, \mathbf{v}^{p}) = G(p)$ we have (cf.~\cite{Casas1985},
 proof of Theorem~3 with piecewise linear finite elements)
 \begin{align*}
  \innerProd{q^{\gamma}}{p}_{L^{2}(\Omega)} =
   \innerProd{q^{\gamma}}{-\Delta y^{p}}_{L^{2}(\Omega)} =
   \innerProd{-\Delta q^{\gamma}}{y^{p}}_{L^{2}(\Omega)} =
   \innerProd{y^{p}}{\nu^{\gamma}}_{L^{2}(\Omega)}
 \end{align*}
 and furthermore, setting~$(y^{p}_{h}, \mathbf{v}^{p}_{h}) = G_{h}(p)$ and using
 the definitions of~$q^{\nu}_{h}$ and~$\mathbf{v}^{\nu}_{h}$,
 \begin{align*}
  \innerProd{q^{\nu}_{h}}{p}_{L^{2}(\Omega)}
   = & -\innerProd{q^{\nu}_{h}}{\divergence \mathbf{v}^{p}_{h}}_{L^{2}(\Omega)}
   = \innerProd{\mathbf{v}^{\nu}_{h} \cdot \mathbf{v}^{p}_{h}}{1}_{L^{2}(\Omega)}
   = -\innerProd{y^{p}_{h}}{\divergence \mathbf{v}^{\nu}_{h}}_{L^{2}(\Omega)}\\
   = & \innerProd{y^{p}_{h}}{\nu^{\gamma}}_{L^{2}(\Omega)}.
 \end{align*}
 Inserting these into~\eqref{eq:q_gamma_times_p}, we conclude
 with~$\norm{p}_{L^{\infty}(\Omega)} \leq 1$ that
 \begin{align*}
 \norm{q^{\gamma} - q^{\nu}_{h}}_{L^{1}(\Omega)}
  = & \innerProd{q^{\gamma} - q^{\nu}_{h}}{p}_{L^{2}(\Omega)}
  = \innerProd{q^{\gamma} - q^{\nu}_{h}}{\sgn(q^{\gamma} - q^{\nu}_{h})}_{L^{2}(\Omega)}
  = \innerProd{y^{p} - y^{p}_{h}}{\nu^{\gamma}}_{L^{2}(\Omega)} \\
  \leq & \norm{y^{p} - y^{p}_{h}}_{L^{\infty}(\Omega)}
   \norm{\nu^{\gamma}}_{L^{1}(\Omega)}
  \leq C h \abs{\log h} \norm{p}_{L^{\infty}(\Omega)}
  \leq C h \abs{\log h}.
 \end{align*}
 
 Finally, we
 set~$(\bar{y}_{h},\bar{\mathbf{v}}_{h}) = G_{h}(z_{h}\,l^{\gamma})$ and
 rewrite the third addend~$(III)$ using the definitions of~$y^{\gamma}_{h}$,
 $\bar{y}_{h}$, $q^{\nu}_{h}$ and~$q^{\gamma}_{h}$ together with the
 monotonicity of the min-function
 \begin{align*}
  & \mathrel{\phantom{=}} (III) = \innerProd{\gamma(y^{\gamma} + \tau)^{-} -
   \gamma(y^{\gamma}_{h} + \tau)^{-}}{y^{\gamma}_{h} - \bar{y}_{h}}_{L^{2}(\Omega)} \\
  & = \underbrace{\innerProd{\gamma(y^{\gamma} + \tau)^{-} - \gamma(y^{\gamma}_{h} +
  \tau)^{-}}{y^{\gamma}_{h} - y^{\gamma}}_{L^{2}(\Omega)}}_{\leq 0}
   + \innerProd{\gamma(y^{\gamma} + \tau)^{-} - \gamma(y^{\gamma}_{h} +
   \tau)^{-}}{y^{\gamma} - \bar{y}_{h}}_{L^{2}(\Omega)} \\
   & \leq \innerProd{\gamma(y^{\gamma} + \tau)^{-} - \gamma(y^{\gamma}_{h} +
   \tau)^{-}}{y^{\gamma} - \bar{y}_{h}}_{L^{2}(\Omega)}.
 \end{align*}
 Now let~$(\bar{y}, \bar{\mathbf{v}}) = G(z_{h}\,l^{\gamma})$ and similar to
 the proof of~\cite[Theorem~3.5]{HintermuellerHinze2009} one has
 for~$0 < h < h_{0}$
 \begin{align*}
  & \mathrel{\phantom{=}} \innerProd{\gamma(y^{\gamma} + \tau)^{-} -
   \gamma(y^{\gamma}_{h} + \tau)^{-}} {y^{\gamma} - \bar{y}_{h}}_{L^{2}(\Omega)} \leq
   \max \left\{ \norm{\nu^{\gamma}}_{L^{1}(\Omega)}, 
   \norm{\nu^{\gamma}_{h}}_{L^{1}(\Omega)} \right\}
    \norm{y^{\gamma} - \bar{y}_{h}}_{L^{\infty}(\Omega)} \\
  & \leq C \norm{y^{\gamma} - \bar{y}}_{L^{\infty}(\Omega)}
   + C \norm{\bar{y} - \bar{y}_{h}}_{L^{\infty}(\Omega)}
   \leq C \norm{y^{\gamma} - \bar{y}}_{H^{2}(\Omega)} + C h \abs{\log h}
   \norm{z_{h}}_{L^{\infty}(\Omega)} \norm{l^{\gamma}}_{L^{\infty}(\Omega)} \\
  & \leq C \norm{l^{\gamma}}_{L^{\infty}(\Omega)}
   \norm{z - z_{h}}_{L^{2}(\Omega)} + C h \abs{\log h}
   \leq C h + C h \abs{\log h} \leq C h \abs{\log h}.
 \end{align*}
 Altogether, we obtain the proposed error bound for the controls.
 The estimate for the states follows similarly to the proof of
 Theorem~\ref{thm:error_l_gamma} and with~\eqref{eq:discretization_error_L2}
 and~\eqref{eq:discretization_error_LInfty}.
 With~$(\bar{y}^{h}, \bar{\mathbf{v}}^{h}) = G(z_{h} \, l^{\gamma}_{h})$
 there holds for sufficiently small~$h > 0$ that
 \begin{align*}
  \norm{y^{\gamma} - y^{\gamma}_{h}}_{L^{\infty}(\Omega)}
  & \leq \norm{y^{\gamma} - \bar{y}}_{L^{\infty}(\Omega)}
  + \norm{\bar{y} - \bar{y}^{h}}_{L^{\infty}(\Omega)}
  + \norm{\bar{y}^{h} - y^{\gamma}_{h}}_{L^{\infty}(\Omega)} \\
  & \leq C \norm{y^{\gamma} - \bar{y}}_{H^{2}(\Omega)}
  + C \norm{\bar{y} - \bar{y}^{h}}_{H^{2}(\Omega)}
  + Ch \abs{\log h} \norm{z_{h}}_{L^{\infty}(\Omega)}
   \norm{l^{\gamma}_{h}}_{L^{\infty}(\Omega)} \\
  & \leq C \norm{l^{\gamma}}_{L^{\infty}(\Omega)} \norm{z - z_{h}}_{L^{2}(\Omega)}
  + C \norm{z_{h}}_{L^{\infty}(\Omega)} \norm{l^{\gamma} - l^{\gamma}_{h}}_{L^{2}(\Omega)}
  + Ch \abs{\log h} \\
  & \leq C h \norm{f}_{L^{2}(\Omega)}
  + C h^{1/2} \abs{\log h}^{1/2}
  + Ch \abs{\log h} \\
  & \leq C h^{1/2} \abs{\log h}^{1/2}.
 \end{align*}
\end{proof}
Combination of the estimates for the two error components immediately gives
a bound for the overall error.
\begin{thm} \label{thm:error_overall}
 Let~$l$ and~$l^{\gamma}_{h}$ denote the solutions of~$\dualProblem$
 and~$\discreteProblemMY$, resp., with corresponding states~$(y,\mathbf{v})$
 and~$(y^{\gamma}_{h} , \mathbf{v}^{\gamma}_{h})$. Then there exist
 an~$h_{0} > 0$ and for every~$\varepsilon >
 0$ a $\gamma$- and~$h$-independent constant~$C_{\varepsilon}$, such that
 for all~$0 < h < h_{0}$ and~$\gamma > 0$
 \begin{align} \label{eq:error_overall}
  & \mathrel{\phantom{\leq}} \norm{l - l^{\gamma}_{h}}_{L^{2}(\Omega)} +
  \norm{y - y^{\gamma}_{h}}_{L^{\infty}(\Omega)}
  \leq
  C_{\varepsilon} \left( \gamma^{-\frac{1}{2} \left(1 - \frac{d}{4}\right) + \varepsilon} +
  h^{\frac{1}{2}} \abs{\log h}^{\frac{1}{2}} \right).
 \end{align}
\end{thm}
\begin{proof}
We split the overall control error into the sum
of~$\norm{l - l^{\gamma}}_{L^{2}(\Omega)}$
and~$\norm{l^{\gamma} - l^{\gamma}_{h}}_{L^{2}(\Omega)}$
and apply
Theorems~\ref{thm:error_l_gamma}
and~\ref{thm:error_l_gamma_h} to estimate the regularization and
discretization errors. For the states we proceed analogously.
\end{proof}
Estimate~\eqref{eq:error_overall} now suggests a coupling of~$\gamma$ and the
grid size~$h$ of the form~$\gamma = O(h^{-\kappa})$. Equilibrating the
errors depending on the dimension~$d$, we obtain
\begin{align*}
  \norm{l - l^{\gamma}_{h}}_{L^{2}(\Omega)} +
  \norm{y - y^{\gamma}_{h}}_{L^{\infty}(\Omega)} \leq
   C_{\varepsilon} h^{\frac{1}{2} - \varepsilon}
  \begin{cases}
   \text{for } \kappa = 2, \text{ if } d = 2, \\
   \text{for } \kappa = 4, \text{ if } d = 3.
  \end{cases}
\end{align*}


\subsection{Variational discretization of~$\regProblemMY$ with continuous
piecewise linear finite elements}
Next we consider piecewise linear and continuous finite element approximations
of the state in problem~$\regProblemMY$. For~$g \in L^2(\Omega)$ given we
denote by~$y_h = G_h(g) \in Y_h$ the solution to
\begin{align*}
 \innerProd{\nabla y_h}{\nabla w_h}_{L^{2}(\Omega)^{d}}
 = \innerProd{g}{w_h}_{L^{2}(\Omega)} \quad \forall w_h \in
 Y_{h},
\end{align*}
where
\begin{align*}
 Y_{h} \defined \set{w_h \in C(\bar{\Omega}) | w_{h|T} \in P^{1}(T)
 \quad \forall T \in \mathcal{T}_{h}} \cap H^1_0(\Omega).
\end{align*}
Then, with~$y = T(g)$ we have the well known error estimate
\begin{align} \label{eq:discretization_error_L2_linear}
 \norm{y - y_h}_{L^{2}(\Omega)} + h \norm{\nabla(y - y_h)}_{L^{2}(\Omega)^{d}} \leq
 C h^2 \norm{y}_{H^{2}(\Omega)} \leq C h^2 \norm{g}_{L^{2}(\Omega)}.
\end{align}
Provided that~$g \in L^{\infty}(\Omega)$, one can use the
$L^{p}(\Omega)$-estimate from~\cite[Theorem~2.2 and the subsequent Remark]{Schatz1998} together with careful control of the constant
in the a-priori $L^{p}(\Omega)$-estimate for~\eqref{eq:dirichlet_diff}
to prove the following bound on the error in the maximum norm
\begin{align} \label{eq:discretization_error_LInfty_linear}
 \norm{y - y_h}_{L^{\infty}(\Omega)} \leq
 C h^2 \abs{\log h}^{2} \norm{g}_{L^{\infty}(\Omega)}.
\end{align}
We refer to~\cite[Lemma~1]{DeckelnickHinze2008} for the complete argument.

Enforcing the state
constraints in the (inner) nodes $\{x_i\}_{i=1,\ldots,N}$ of the grid, i.e.,
using the space~$Y_{\mathrm{ad},h} \defined \Set{\phi_{h} \in Y_{h} |
\phi_{h}(x_i) \geq -\tau, \quad 1 \leq i \leq N}$ of admissible states,
the variational discretization of the regularized optimization
problem~$\regProblemMY$ is now straightforward. In
problem~$\discreteProblemMY$ we only have to replace the discrete solution
operator by its counterpart of the present solution. With this and the implicit
datum~$z_{h} = G_{h}(f)$ we obtain
\begin{align*}
 \min_{l \in L^{\infty}(\Omega)} J^{\gamma}_{h}(l)
  \defined \int_{\Omega}{l^{-1/3} \ud x} + \frac{\gamma}{2} \int_{\Omega}
  \big( (y_{h} + \tau)^{-} \big)^{2} \ud x
 \tag{$\discretePwLinearProblemMYTag$}
\end{align*}
subject to
\begin{align*}
 &&&&&& y_{h}  & = G_{h}(z_{h} \, l),      &&&&& \\
 &&&&&& l      & \in L_{\mathrm{ad}}.
\end{align*}
Problem~$\discretePwLinearProblemMY$ admits a unique
solution~$l^{\gamma}_{h}$, which together with the discrete adjoint
state~$q^{\gamma}_{h}$ satisfies
\begin{align}
 &&&& y^{\gamma}_{h} & = G_{h}(z_{h} \, l^{\gamma}_{h}),      &&&&&
  \label{os_my_h:prim_eq_p1} \\
 &&&& q^{\gamma}_{h} & = G_{h} \left( \gamma(y^{\gamma}_{h} + \tau)^{-} \right),
  \label{os_my_h:adj_eq_p1} \\
  &&&& \int_{\Omega} \left( q^{\gamma}_{h}\, z_{h} - \frac{1}{3} \,
   (l^{\gamma}_{h})^{-4/3} \right) (k - l^{\gamma}_{h}) \ud x
   & \geq 0 \qquad \forall \, k \in L_{\mathrm{ad}},
   \label{os_my_h:var_ineq_p1} \\
  &&&& l^{\gamma}_{h}                 & \in L_{\mathrm{ad}}.
   \nonumber
\end{align}
The variational inequality~\eqref{os_my_h:var_ineq_p1} can again be rewritten
as a projection formula
\begin{align} \label{os_my_h:proj_p1}
 l^{\gamma}_{h} = \left( P_{\left[m^{4},M^{4}\right]} \left( 3\,q^\gamma_{h}\,z_{h}
 \right) \right)^{-3/4}.
\end{align}
We note that~$(q_{h} z_{h})_{|T}$ is a quadratic function, whose projection in
general can not be represented by a polynomial over~$T$.
As in Lemma~\ref{lem:nu_gamma_h} we infer~$\norm{\nu^\gamma_h
 \defined \gamma(y^{\gamma}_{h} + \tau)^{-}}_{L^{1}(\Omega)}
\leq C$ independently of~$\gamma$ and~$h$. Moreover,
Theorem~\ref{thm:error_l_gamma_h} holds accordingly.
\begin{thm} \label{thm:error_l_gamma_h_p1}
 Let~$l^{\gamma}$ and~$l^{\gamma}_{h}$ be the solutions of~$\regProblemMY$
 and~$\discretePwLinearProblemMY$, resp., with corresponding
 states~$y^{\gamma}$ and~$y^{\gamma}_{h}$. Then there exists an~$h_{0} > 0$ and
 a~$\gamma$- and~$h$-independent positive constant~$C$, such that for
 all~$h \in (0,h_{0})$ and all~$\gamma > 0$ we have
 \begin{align*}
  \norm{l^{\gamma} - l^{\gamma}_{h}}_{L^{2}(\Omega)} +
  \norm{y^{\gamma} - y^{\gamma}_{h}}_{H^{1}(\Omega)} +
  \norm{y^{\gamma} - y^{\gamma}_{h}}_{L^{\infty}(\Omega)} \leq C h \abs{\log h}.
 \end{align*}
\end{thm}
\begin{proof}
 With the adjoint states~$q^{\gamma}$ and~$q^{\gamma}_{h}$ 
 from~\eqref{os_my:adj_eq_weak} and~\eqref{os_my_h:adj_eq_p1}, resp.,
 and~$q^{\nu}_{h} = G_{h}(\nu^{\gamma})$, we obtain
 as in the proof of Theorem~\ref{thm:error_l_gamma_h}
 \begin{align*}
  &\phantom{\ =\ } C \norm{l^{\gamma} - l^{\gamma}_{h}}_{L^{2}(\Omega)}^{2}
  \leq \innerProd{q^{\gamma}z -
  q^{\gamma}_{h} z_{h}}{l^{\gamma}_{h} - l^{\gamma}}_{L^{2}(\Omega)} \\
  & = \innerProd{q^{\gamma}z - q^{\gamma} z_{h}}
  {l^{\gamma}_{h} - l^{\gamma}}_{L^{2}(\Omega)} +
  \innerProd{q^{\gamma}z_{h} - q^{\nu}_{h} \, z_{h}}
  {l^{\gamma}_{h} - l^{\gamma}}_{L^{2}(\Omega)} +
  \innerProd{q^{\nu}_{h} z_{h} - q^{\gamma}_{h} \, z_{h}}
  {l^{\gamma}_{h} - l^{\gamma}}_{L^{2}(\Omega)} \\
  & \defines (I) + (II) + (III),
 \end{align*}
 where, using Lemma~\ref{lem:my_bounds},
 \eqref{eq:discretization_error_L2_linear},
 \eqref{eq:discretization_error_LInfty_linear}
 and~$\bar{y} = T(z_h l^\gamma)$, $\bar{y}_{h} = T_h(z_h l^\gamma)$
 (cf. proof of Theorem~\ref{thm:error_l_gamma_h}), we now can estimate
 the three addends as follows
 \begin{align*}
  (I) & \leq \norm{q^{\gamma}}_{L^{2}(\Omega)} \norm{z - z_{h}}_{L^{2}(\Omega)}
   \norm{l^{\gamma} - l^{\gamma}_h}_{L^{\infty}(\Omega)} \leq C h^2,
   \nonumber \\ 
  (II) & \leq \norm{z_{h}}_{L^{\infty}(\Omega)} \norm{q^{\gamma} - q^{\nu}_{h}}_{L^{1}(\Omega)}
   \norm{l^{\gamma} - l^{\gamma}_{h}}_{L^{\infty}(\Omega)} \leq C h^2\abs{\log h}^2,
   \label{eq:ii} \\
  (III) & \leq C\norm{z - z_h}_{L^{2}(\Omega)} + C\norm{\bar{y} - \bar{y}_{h}}_{L^{\infty}(\Omega)}
   \leq C \left( h^2 \norm{f}_{L^{2}(\Omega)} +
   h^2 \abs{\log h}^2 \norm{z_h l^\gamma}_{L^{2}(\Omega)} \right)\nonumber \\
   & \leq C h^2 \abs{\log h}^2,\nonumber
 \end{align*}
 with~$(III)$ and the estimate
 of~$\norm{q^{\gamma} - q^{\nu}_{h}}_{L^{1}(\Omega)}$ deduced
 similarly as in the proof of Theorem~\ref{thm:error_l_gamma_h}.
 Combining the above estimates completes the proof for the control error. The bounds on the
state errors can be obtained using similar arguments as for
Theorem~\ref{thm:error_l_gamma_h}.
\end{proof}
The estimate for the overall error follows from combining
Theorems~\ref{thm:error_l_gamma} and~\ref{thm:error_l_gamma_h_p1}.
\begin{thm} \label{thm:error_overall_p1}
 Let~$l$ and~$l^{\gamma}_{h}$ be the solutions of~$\dualProblem$
 and~$\discretePwLinearProblemMY$, resp., with corresponding states~$y$
 and~$y^{\gamma}_{h}$. Then there exist
 an~$h_{0} > 0$ and for every~$\varepsilon > 0$ a~$\gamma$-
 and~$h$-independent constant~$C_{\varepsilon}$, such that for
 all~$h \in (0,h_{0})$ and all~$\gamma > 0$ we have
 \begin{align*}
  \norm{l - l^{\gamma}_{h}}_{L^{2}(\Omega)} +
  \norm{y - y^{\gamma}_{h}}_{H^{1}(\Omega)} +
  \norm{y - y^{\gamma}_{h}}_{L^{\infty}(\Omega)} \leq C_{\varepsilon}
  \left( \gamma^{-\frac{1}{2} \left(1 - \frac{d}{4}\right) + \varepsilon} +
  h \abs{\log h} \right).
 \end{align*}
 Coupling~$\gamma$ and~$h$ again in the form~$\gamma = O(h^{-\kappa})$
 to balance the error contributions, we obtain
 \begin{align*}
  \norm{l - l^{\gamma}_{h}}_{L^{2}(\Omega)} +
  \norm{y - y^{\gamma}_{h}}_{H^{1}(\Omega)} +
  \norm{y - y^{\gamma}_{h}}_{L^{\infty}(\Omega)} \leq C_{\varepsilon} h^{1 - \varepsilon}
  \begin{cases}
   \text{for } \kappa = 4, \text{ if } d = 2, \\
   \text{for } \kappa = 8, \text{ if } d = 3.
  \end{cases}
 \end{align*}
\end{thm}
It turns out that variational discretization with piecewise linear, continuous
finite elements yields a better approximation order for our optimal control
problem than variational discretization with lowest-order Raviart-Thomas mixed
finite elements. In order to achieve this, however, a more progressive
coupling of regularization parameter and grid size is necessary in the case
of piecewise linear, continuous finite elements. While this would mean a
better approximation of the unregularized solution, it can be a drawback
if at the same time the problems become harder to solve, cf. the numerical
examples in the following section.
\section{Numerical examples} \label{sec:num}
To approximate the solution of problem~$\dualProblem$ we apply a path-following
algorithm in the parameter~$\gamma$, which is sent to~$\infty$
(cf., e.g.,~\cite{HintermuellerKunisch2006_1}). The
subproblems~$\discreteProblemMY$ and~$\discretePwLinearProblemMY$, resp., are
solved with a semismooth Newton method, where the parameter~$\gamma$ is coupled
to~$h$ according to Theorems~\ref{thm:error_overall}
and~\ref{thm:error_overall_p1}, respectively. The semismooth Newton method is
described in the Appendix. We conclude this work with two example problems to
supplement our numerical analysis.

\textit{Example 1:}
In order to construct an example problem with known solution
(cf. Section~2.9 in \cite{Troeltzsch2010}), we add the term
\begin{align*}
 (\alpha / 2)\norm{T(z\,l) - y_{\Omega}}_{L^{2}(\Omega)}^{2}
\end{align*}
to the cost functional, where~$\alpha > 0$ and~$y_{\Omega} \in L^{2}(\Omega)$,
and change the state equation to
\begin{align*}
 -\Delta y = z \, l + e_{\Omega},
\end{align*}
with a suitable function~$e_{\Omega}  \in L^{2}(\Omega)$, see below.
To construct the exact solution we set~$\Omega = (0,1)^{2}$ and
define~$r(x) \defined \abs{x - \bar{x}}$, where~$\bar{x} = (1/2, 1/2)$.
We choose~$\tau = 0.1$, $m = 0.35$, $M = 0.45$
and~$\alpha = 1$, and define the (optimal) state
\[ y(r) =
   \begin{cases}
    -0.1,    & r \leq \frac{1}{8}, \\
    614.4 r^5 - 768 r^4 + 352 r^3 - 72 r^2 + \frac{27}{4} r - \frac{27}{80},
             & r \in (\frac{1}{8},\frac{3}{8}), \\
    0,       & r \geq \frac{3}{8}
   \end{cases}
\]
and the adjoint state
\[ q(r)  =
    \begin{cases}
     -r^2 + \frac{1}{64},    & r <    \frac{1}{8}, \\
     0,                      & r \geq \frac{1}{8}.
    \end{cases}
\]
The multiplier~$\nu$ is composed of a regular part concentrated
in~$\Omega_1 \defined B(\bar{x},1/8)$, and a part concentrated on the
boundary~$\partial \Omega_{1}$. Taking
\[ y_{\Omega}(r) =
   \begin{cases}
    -5.1,  & r <    \frac{1}{8}, \\
    y(r),  & r \geq \frac{1}{8},
   \end{cases}
\]
we obtain as action of~$\nu \in \mathscr{M}(\bar{\Omega})$ applied to an
element~$g \in C(\bar{\Omega})$
\[ \int_{\bar{\Omega}} g \ud \nu = -\int_{\Omega_{1}} g \ud x -
 \frac{1}{4} \int_{\partial \Omega_{1}} g \ud s. \]
The auxiliary state~$z$ and the corresponding load~$f$ are set to
\begin{align*}
 z(x) &= \sin(\pi x_{1}) \sin(\pi x_{2}), \quad \text{and}\\
 f(x) &= -\Delta z(x) = 2\pi^{2} \sin(\pi x_{1}) \sin(\pi x_{2}).
\end{align*}
The optimal control~$l$ is given by
\[ l(x) = \Big( P_{\left[m^4,M^4\right]} \big(3(q \circ r)(x)z(x)\big) \Big)^{-3/4}, \]
and
\[ e_{\Omega}(x) \defined -\Delta (y \circ r)(x) - z(x)l(x). \]
\ifpdf
 \begin{figure}[pt] 
  \centering
  \includegraphics[width=5cm]{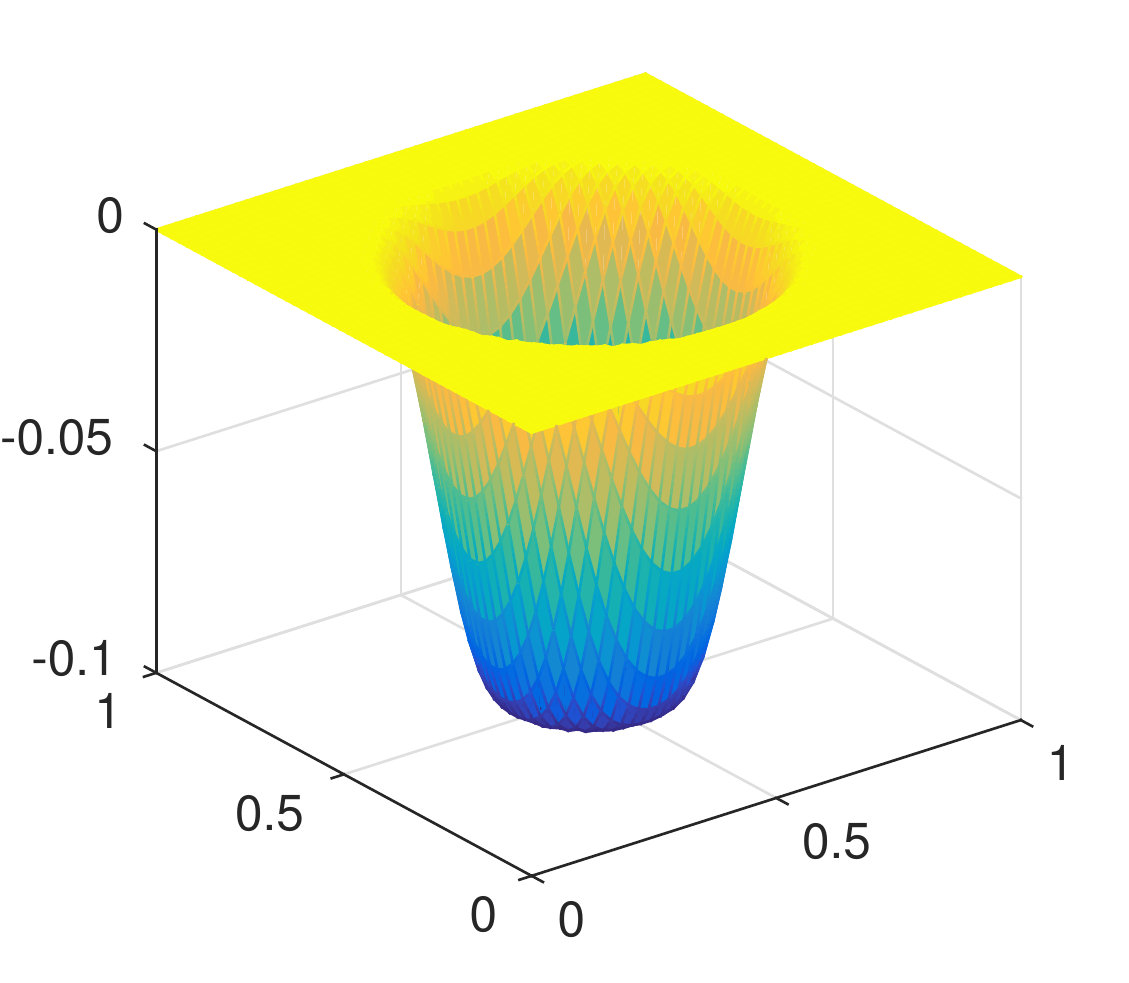}
  \includegraphics[width=5cm]{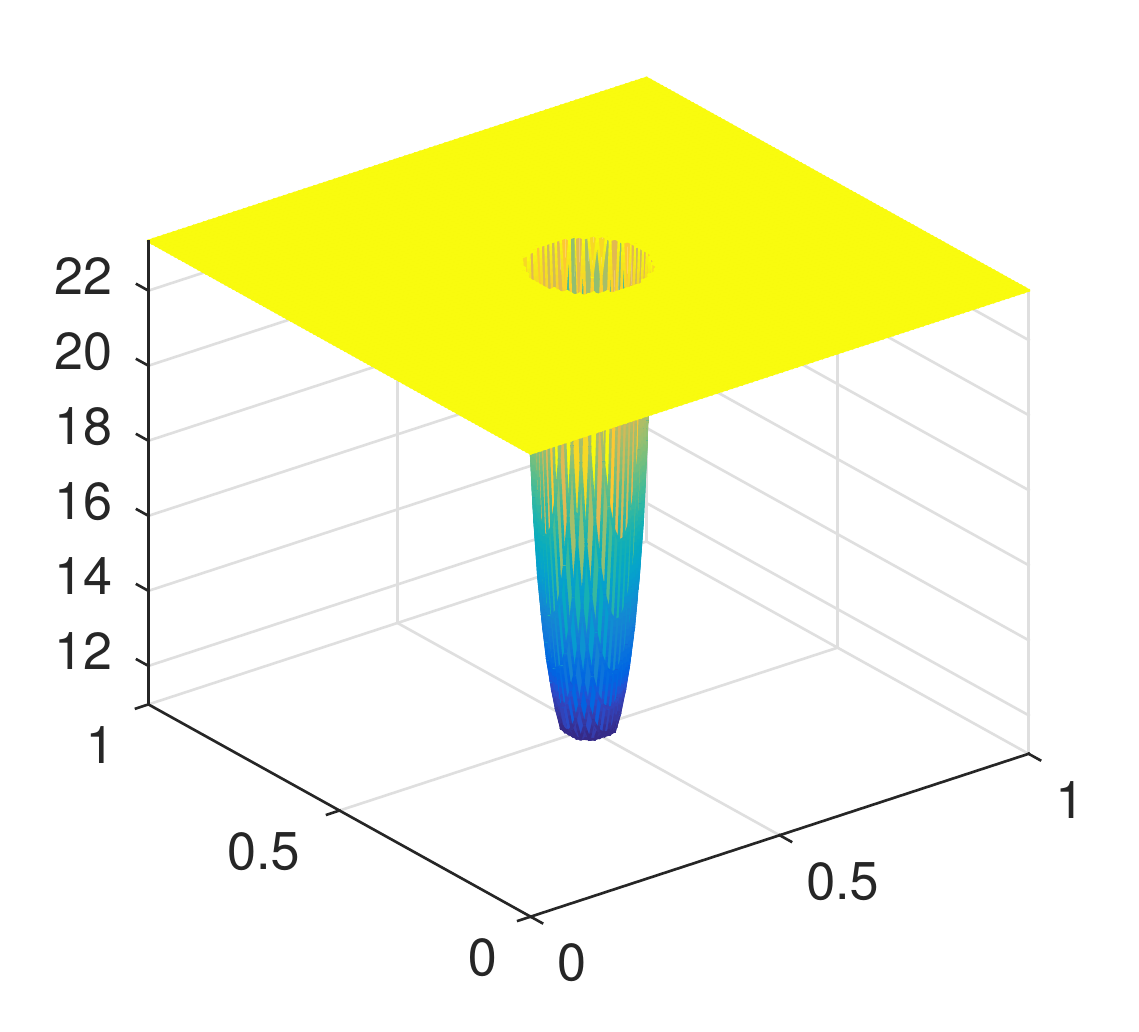}
  \includegraphics[width=5cm]{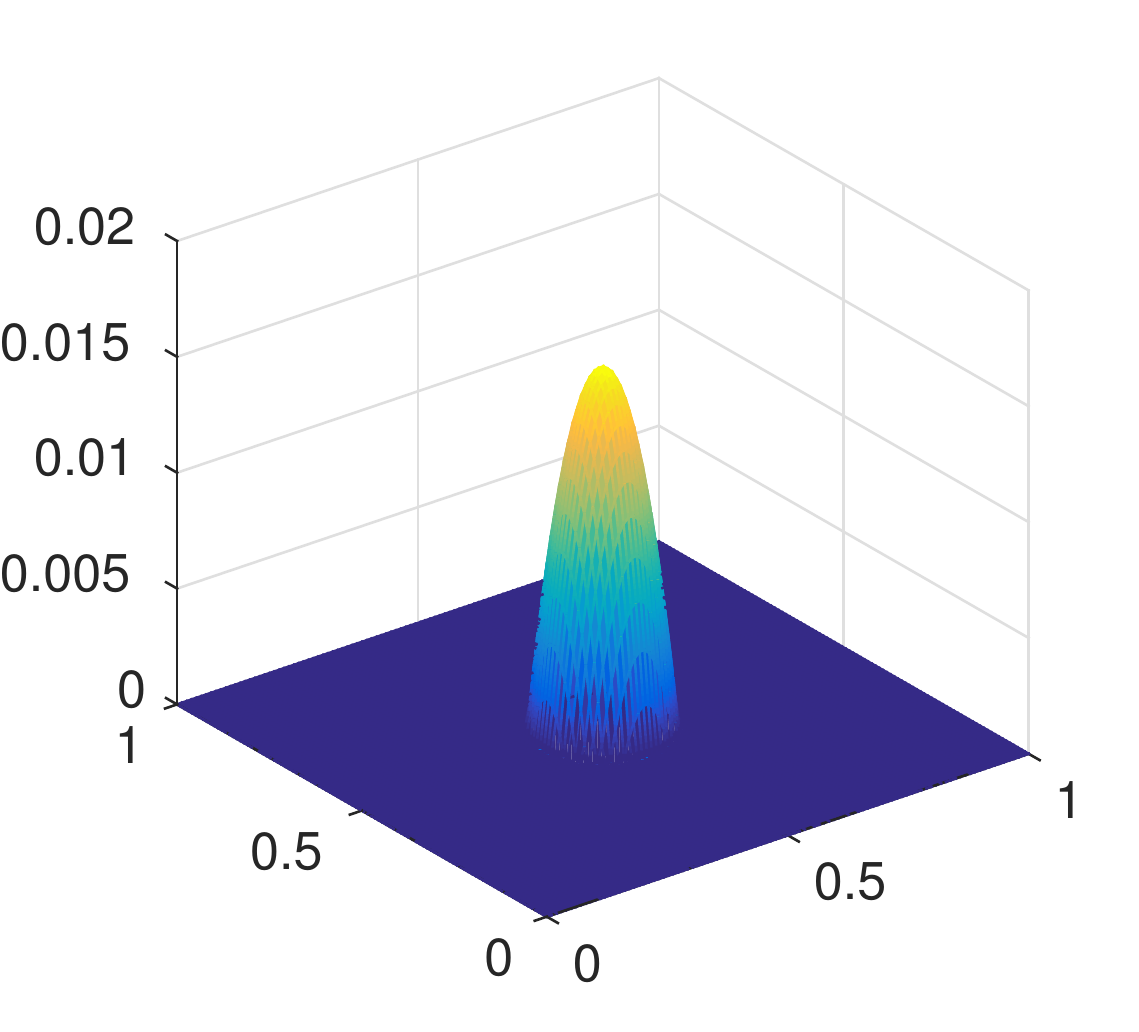}
  \\
  \includegraphics[width=5cm]{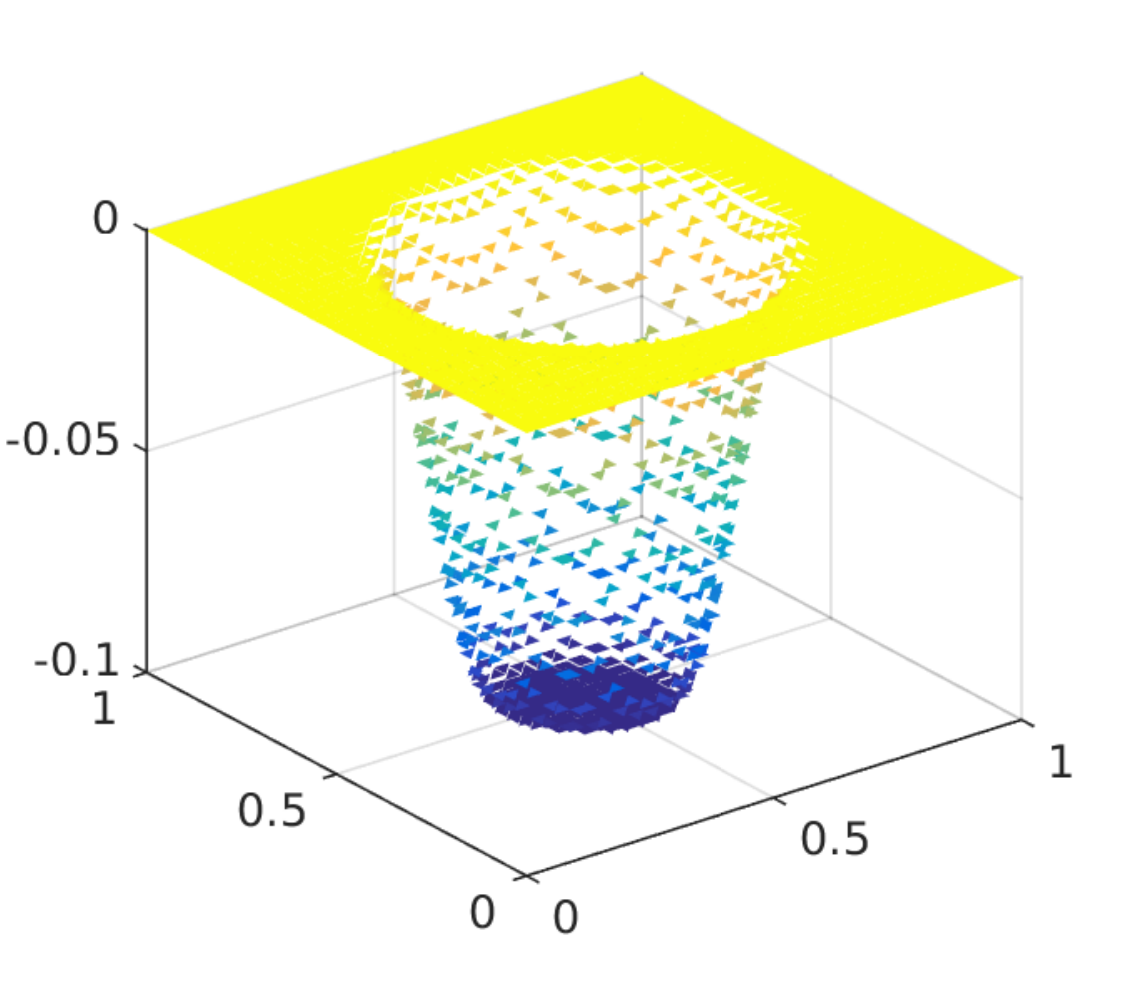}
  \includegraphics[width=5cm]{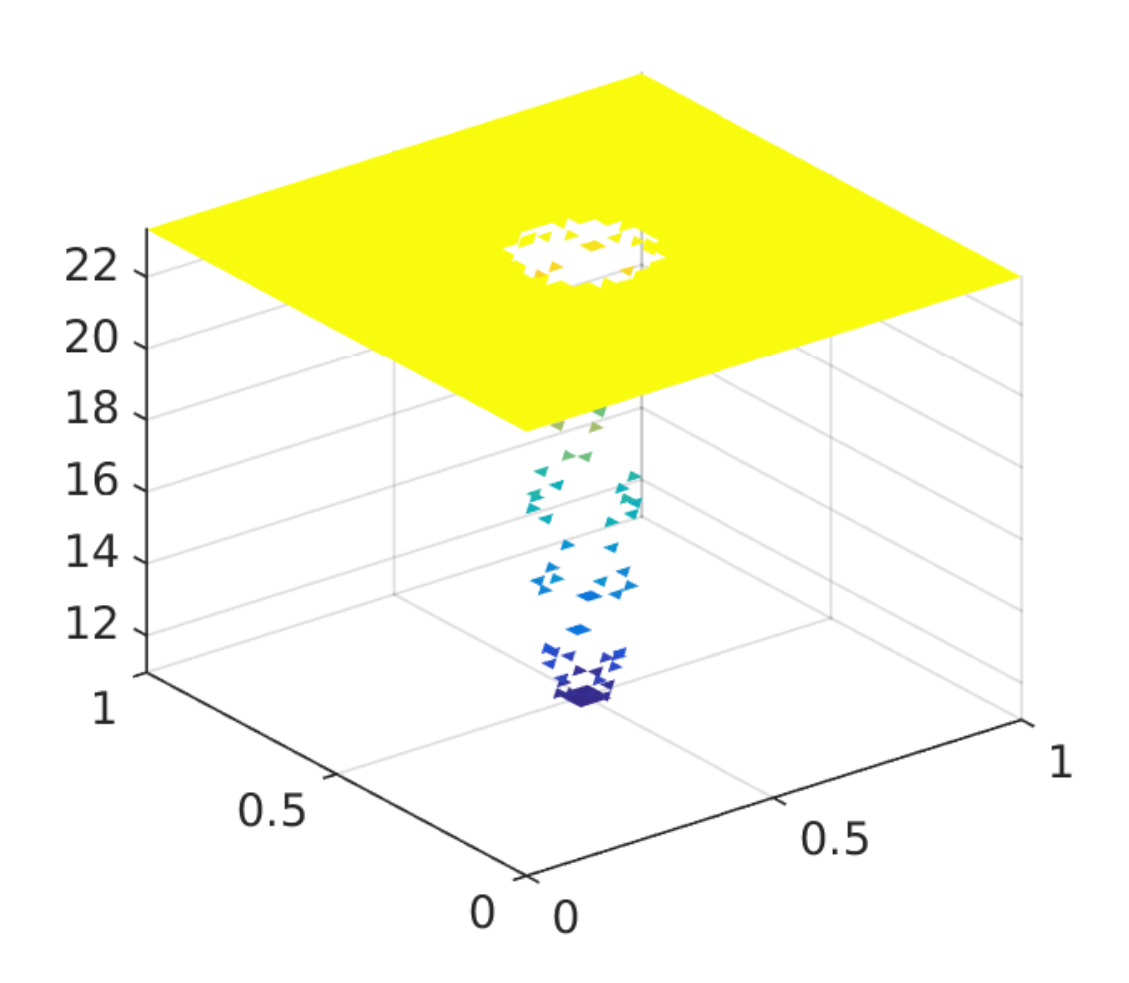}
  \includegraphics[width=5cm]{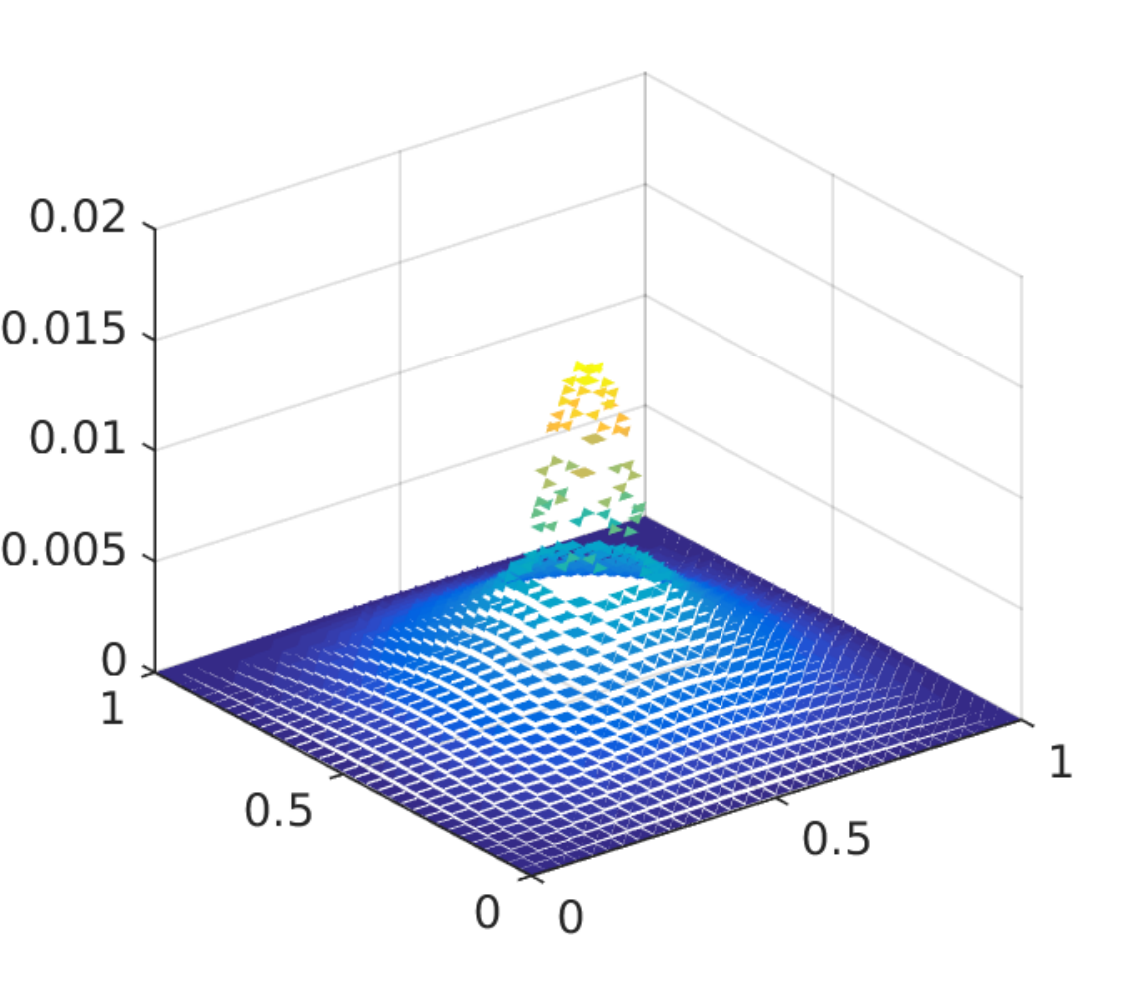}
  \caption{Optimal (top) and optimal discrete piecewise constant (bottom) states~$y$,
   $y_{h}^{\gamma}$ (left),
   controls~$l$, $l_{h}^{\gamma}$ (middle) and adjoint states~$q$, $q_{h}^{\gamma}$
   (right) with~$\gamma = 10^{8}$ and~$h = h_6$ in Example 1.}
  \label{fig:solutions_example_1}
 \end{figure}
\fi
The control and state variables are depicted in
Figure~\ref{fig:solutions_example_1}, where we set
\begin{align*}
 h_k = \sqrt{2} \cdot 2^{1 - k}, \quad k>1.
\end{align*}
In all our numerical tests we approximate integrals using a
3-point Gauß quadrature rule. In the case of piecewise linear elements, however,
we distinguish between those parts of the triangles on which the involved
projections are active and those parts where they are inactive.
\ifpdf
 \begin{figure}[t] 
  \centering
  \includegraphics[height=6cm]{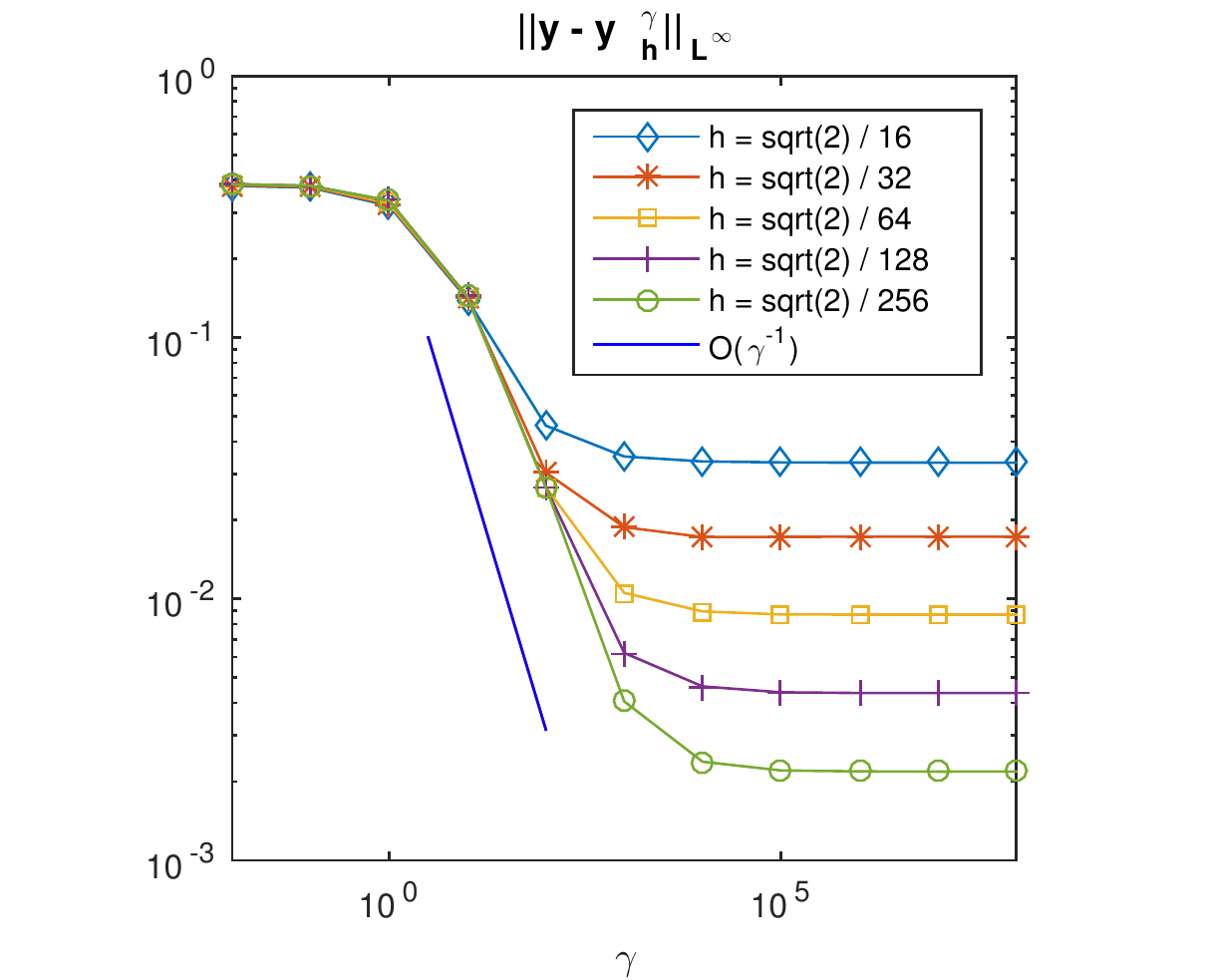}
  \includegraphics[height=6cm]{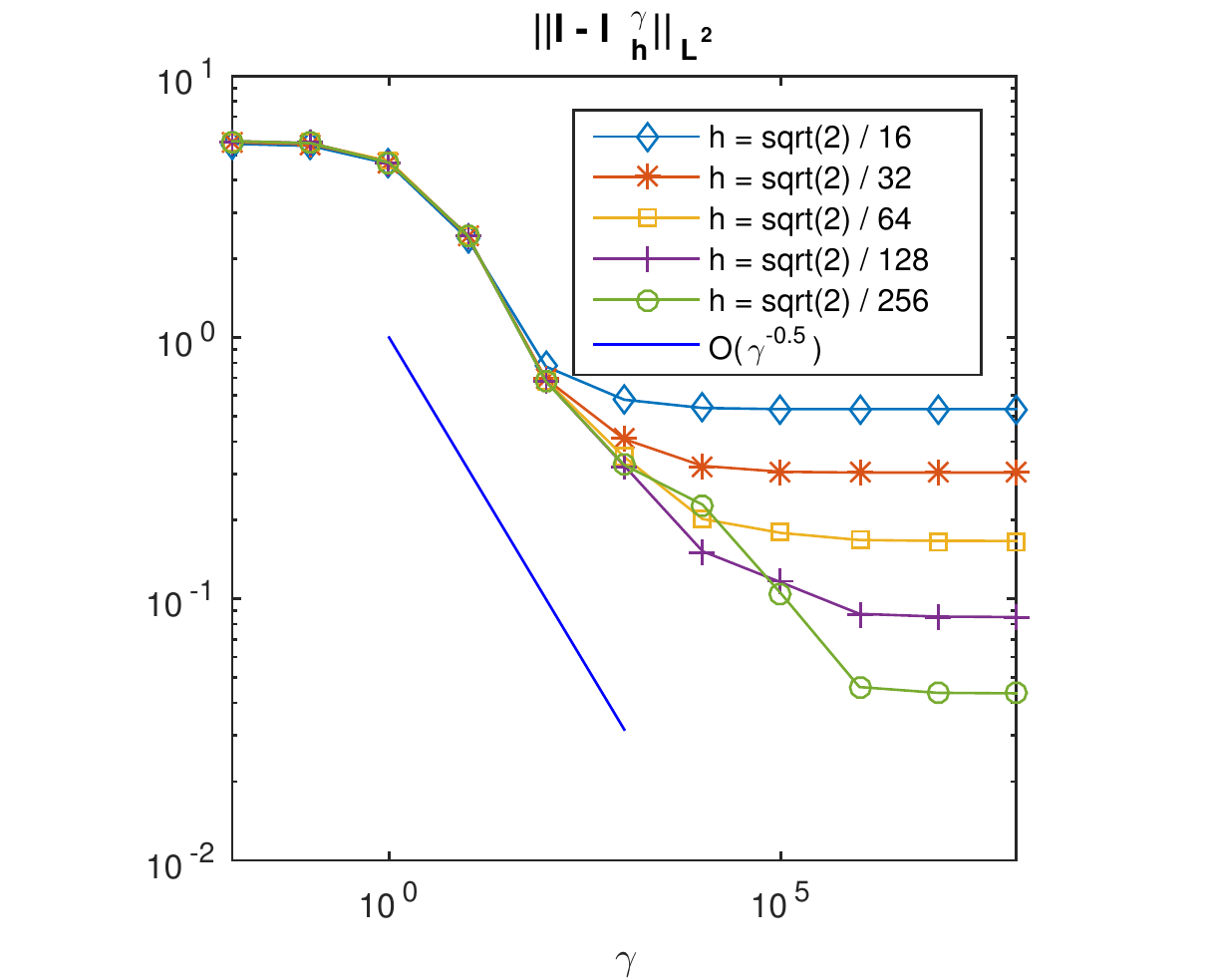}
  \caption{Errors in the states~$y_{h}^{\gamma}$ (left) and
   controls~$l_{h}^{\gamma}$ (right) plotted
   against~$\gamma$ for different values of~$h$ in Example~1 with piecewise
   constant state approximation.}
  \label{fig:test_errors_example_1}
 \end{figure}
\fi

Let us begin with the mixed state approximation.
Table~\ref{tab:errors_example_1} contains $L^\infty(\Omega)$-errors of the 
state and $L^2(\Omega)$-errors of the control variables for a run of the
path-following Algorithm~\ref{alg:path-following},
using the coupling~$\gamma = O(h^{-2})$, starting with~$\gamma_4 = 400$
at~$h = h_4$. We observe errors of the size~$O(h)$, or
equivalently~$O(\gamma^{-1/2})$, which is twice the rate
predicted by Theorem~\ref{thm:error_overall_p1}.
Figure~\ref{fig:test_errors_example_1} shows the development of
the errors in the state and the control variables, respectively,
over a large range of regularization parameters. It can be seen from the
graphs that the discretization error for both variables is 
approximately of order one, which explains the above convergence rate of
size~$O(h)$, and is the square of
the expected error bound derived in the previous section.
Furthermore, the regularization error appears to be of
order~$O(\gamma^{-1})$ for the states and between~$O(\gamma^{-0.3})$
and~$O(\gamma^{-0.5})$ for the controls, while
Theorem~\ref{thm:error_l_gamma} predicts an order of~$O(\gamma^{-0.25})$.
Computation on a series of random
unstructured meshes with grid sizes in the range of the considered uniform
meshes results in the same convergence rates, and thus rules out
superconvergence effects. The observed convergence rates might be
explained by the high regularity of the constructed solution.
\begin{table}[t]
 \caption{Errors and corresponding experimental orders of
  convergence~EOC for the piecewise constant states~$y_{h}^{\gamma}$ and
  controls~$l_{h}^{\gamma}$ in Example 1. 
  Parameters are coupled via~$\gamma = O(h^{-2})$,
  $\gamma_4 = 400$.}
 \begin{center}
  \begin{tabular}[\textwidth]{ccccccccccc}\toprule
   $h_k$ && $\norm{y - y^\gamma_h}_{L^\infty}$ & $\mathrm{EOC}_{y}$ &
   $\norm{l - l^\gamma_h}_{L^2}$ & $\mathrm{EOC}_{l}$ \\ \midrule
   $h_5$ && 3.45e$-$2         &  --       & 5.66e$-$1         &  --  \\
   $h_6$ && 1.74e$-$2         & 0.99      & 3.32e$-$1         & 0.77 \\
   $h_7$ && 8.80e$-$3         & 0.98      & 1.82e$-$1         & 0.87 \\
   $h_8$ && 4.39e$-$3         & 1.00      & 9.47e$-$2         & 0.94 \\
   $h_9$ && 2.19e$-$3         & 1.00      & 4.83e$-$2         & 0.97 \\ \bottomrule
  \end{tabular}
 \end{center}
 \label{tab:errors_example_1}
\end{table}
\ifpdf
 \begin{figure}[t] 
  \centering
  \includegraphics[height=6cm]{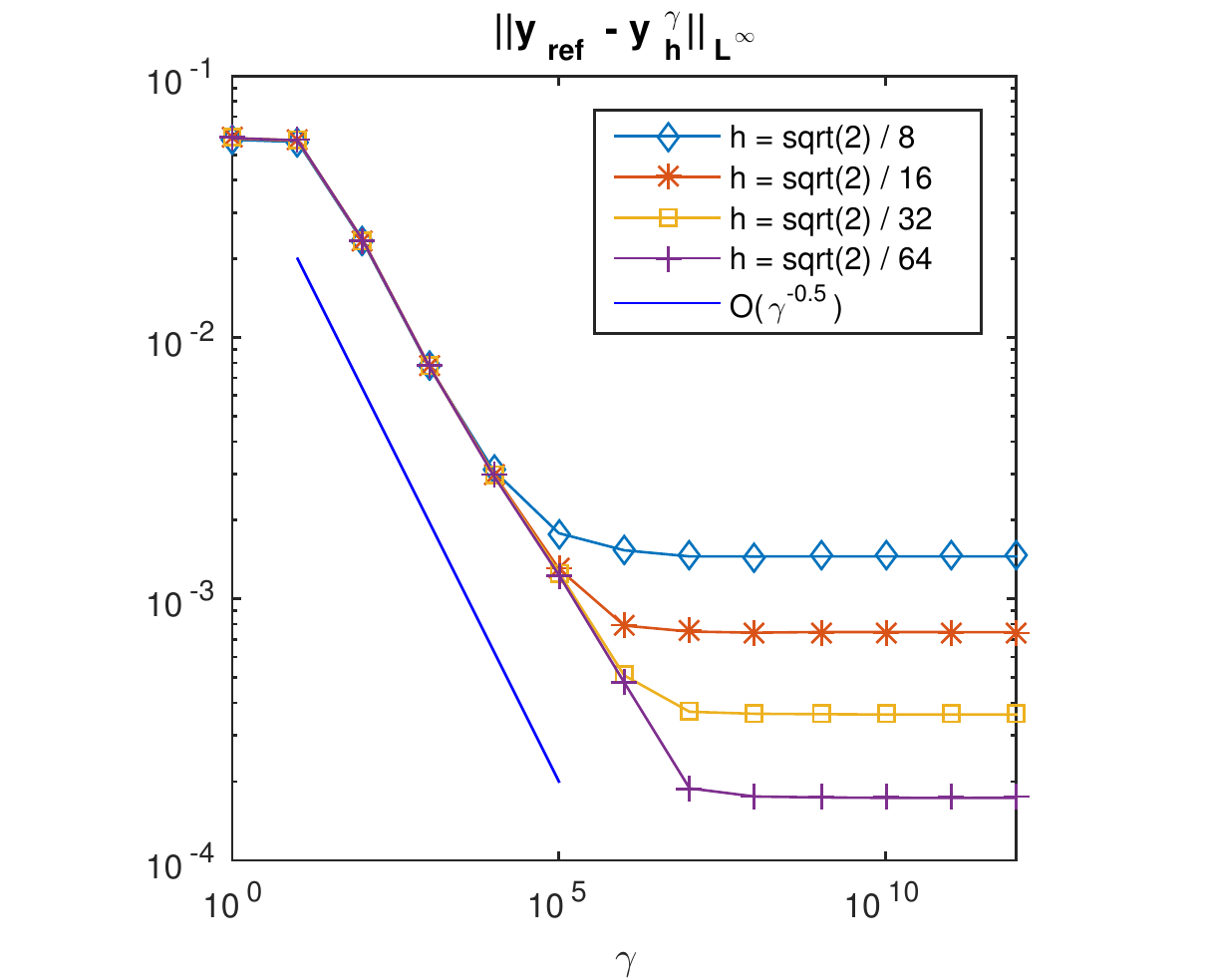}
  \includegraphics[height=6cm]{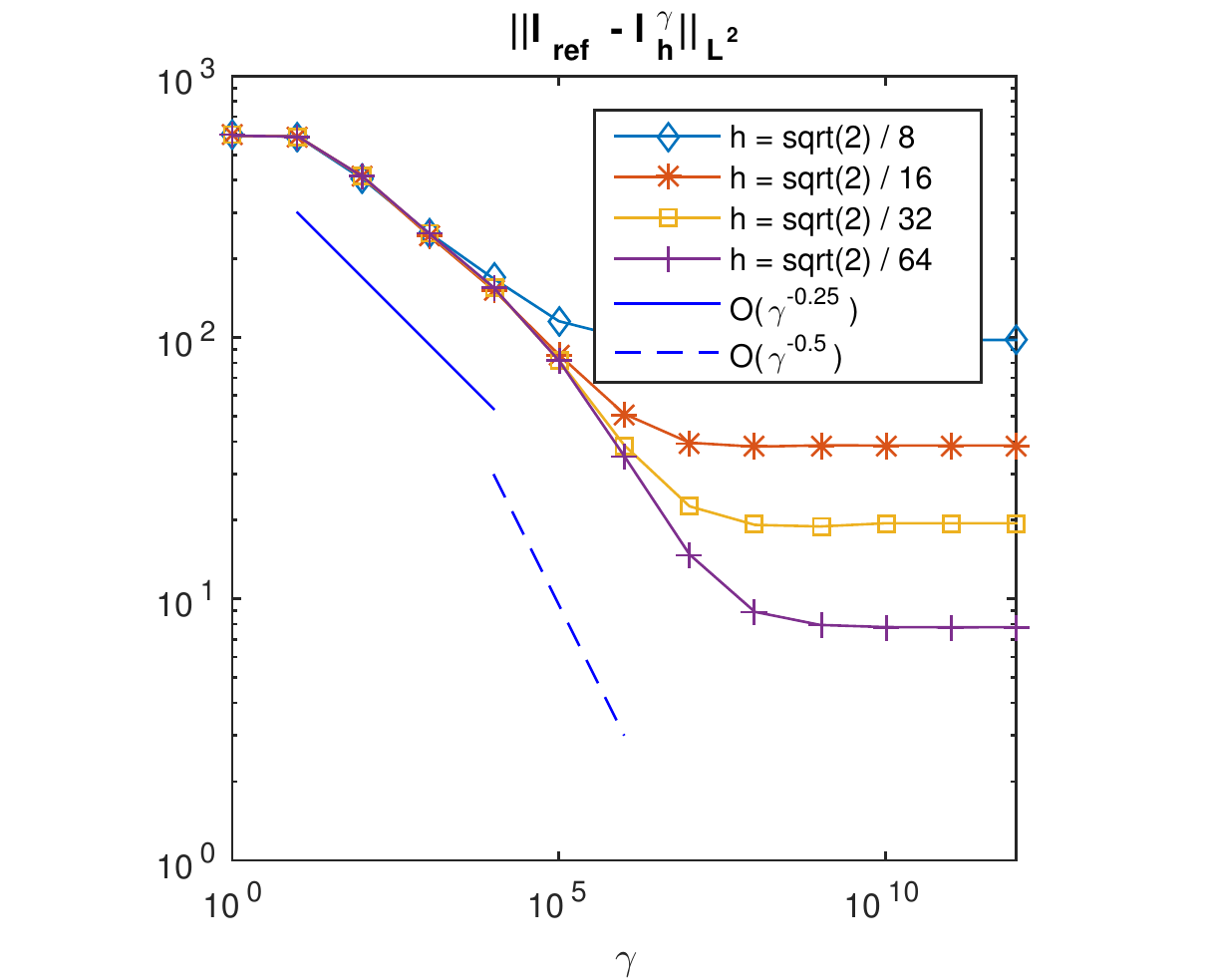}
  \caption{Errors in the states~$y_{h}^{\gamma}$ (left) and
   controls~$l_{h}^{\gamma}$ (right) plotted
   against~$\gamma$ for different values of~$h$ in Example~2 with piecewise
   constant state approximation.}
  \label{fig:test_errors_example_2}
 \end{figure}
\fi
\begin{table}[t]
 \caption{Errors relative to the reference solution~$(y_{\text{ref}},l_{\text{ref}})$ on grid~$h_{9}$
  and corresponding experimental orders of
  convergence for the states~$y_{h}^{\gamma}$ and
  controls~$l_{h}^{\gamma}$ in Example~2.
  Parameters are coupled via~$\gamma = O(h^{-2})$,
  $\gamma_4 = 400$, and~$\gamma = O(h^{-4})$, $\gamma_4 = 16$,
  in the mixed and the piecewise linear, continuous case, resp.}
 \begin{center}
  \begin{tabular}[\textwidth]{ccccccccccccc}\toprule
   && \multicolumn{2}{c}{p.w.\ constant ansatz}
   && \multicolumn{3}{c}{p.w.\ linear ansatz} \\
   \cmidrule(lr){3 - 4} \cmidrule(lr){6 - 8}
   $h_k$ && $\norm{y_{\text{ref}} - y^\gamma_h}_{L^\infty}$ &
   $\norm{l_{\text{ref}} - l^\gamma_h}_{L^2}$
    && $\norm{y_{\text{ref}} - y^\gamma_h}_{L^\infty}$ &
    $\norm{y_{\text{ref}} - y^\gamma_h}_{H^1}$
    & $\norm{l_{\text{ref}} - l^\gamma_h}_{L^2}$ \\ \midrule
   $h_4$ && 1.22e$-$2         & 3.04e+2
         &&         3.82e$-$2 &         9.95e$-$2 &         5.81e+2 & \\
   $h_5$ && 6.30e$-$3         & 2.24e+2
         &&         1.42e$-$2 &         3.42e$-$2 &         3.61e+2 &       \\
   $h_6$ && 3.54e$-$3         & 1.71e+2
         &&         3.80e$-$3 &         1.21e$-$2 &         2.01e+2 &       \\
   $h_7$ && 2.09e$-$3         & 1.23e+2
         &&         1.29e$-$3 &         4.58e$-$3 &         9.81e+1 &       \\
   \cmidrule(lr){3 - 4} \cmidrule(lr){6 - 8}
         && 0.96 & 0.44 && 1.43 & 1.54 & 0.69 \\
         && 0.83 & 0.39 && 1.90 & 1.50 & 0.85 \\
         && 0.76 & 0.48 && 1.56 & 1.40 & 1.03 \\
   \bottomrule
  \end{tabular}
 \end{center}
 \label{tab:errors_example_2}
\end{table}

With the path-following Algorithm~\ref{alg:path-following}
in the appendix 3 to 4 Newton steps are needed to compute the
numerical solution, where we use the tolerance~$10^{-3}$. This result is
achieved independent of the grid size of the underlying mesh and thus
indicates mesh-independence of the algorithm.

Using piecewise linear, continuous state approximations,
however, we in this example are not able to sufficiently progress in the
regularization parameter. Even for small values of~$\gamma$
and with damped Newton steps the semismooth Newton iteration failed to 
converge. This may be
due to large slopes contributed by the term~$e_{\Omega}$ and by the jump
in~$y_{\Omega}$. Similar observations are reported
in~\cite{GuentherHinze2011}, where an interior point solver is used to
treat gradient constraints in elliptic optimal control.
There a jump in the exact control leads to oscillations in the discrete
approximation with piecewise linear, continuous finite elements,
which results in convergence problems for the solver.

\textit{Example 2:}
Here we consider problem~$\primalProblem$ with parameters~$\tau = 0.01$,
$m = 0.1$ and~$M = 0.2$, so that the physical assumptions of a thin plate are
satisfied. We again set~$\Omega = (0,1)^{2}$ and define the load
\begin{align*}
 f(x_{1},x_{2}) \defined
 \begin{cases}
  -0.04, & x_{1} \leq \frac{1}{2} \\
  \phantom{-}0.01, & x_{1} > \frac{1}{2}.
 \end{cases}
\end{align*}
We consider the numerical solutions for~$h=h_{9}$ as reference solutions.
  
Figure~\ref{fig:test_errors_example_2} shows the behaviors
of the state and control errors in the piecewise constant case.
The large magnitudes of the errors in the dual control~$l \defined u^{-3}$
stem from the small bounds on the primary control~$u$.
In this example the regularization error in the
control exhibits two consecutive convergence behaviors: Up
to~$\gamma = 10^{5}$ we observe the theoretically derived
order~$O(\gamma^{-0.25})$, which for larger values of~$\gamma$ improves
to~$O(\gamma^{-0.5})$.
In this example the piecewise linear, continuous state approximation works well with our path-following Algorithm~\ref{alg:path-following}.
In Table~\ref{tab:errors_example_2} we report our numerical findings
for the state and control variables
with piecewise constant and piecewise linear and continuous state
approximation, resp.
The parameters are coupled according to Theorems~\ref{thm:error_overall}
and~\ref{thm:error_overall_p1}, i.e., $\gamma = O(h^{-2})$,
 starting with $\gamma = 400$ on~$h_4$, and~$\gamma = O(h^{-4})$, starting with $\gamma = 16$ on~$h_4$.
With this coupling, the errors in the controls are roughly of the predicted orders, whereas the
state errors seem to converge at a faster rate. Large slopes now occur in the dual control~$l$, due to the small
control constraints on~$u$ and the asymptotically singular behaviour of the adjoint
state near the state active set, which reduces to a point in the limit~$\gamma \to \infty$.
\ifpdf
 \begin{figure}[tp] 
  \centering
  \includegraphics[width=7cm]{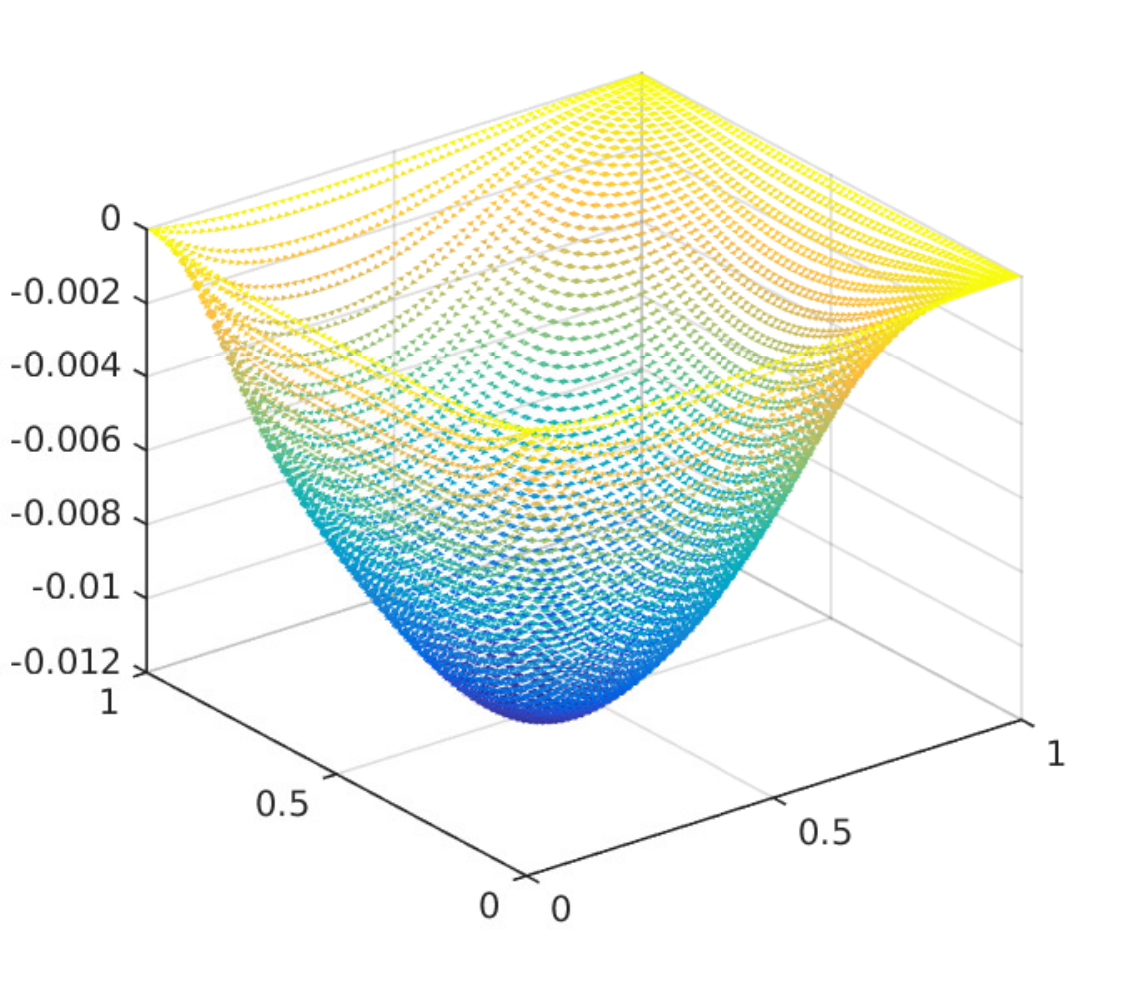}
  \includegraphics[width=7cm]{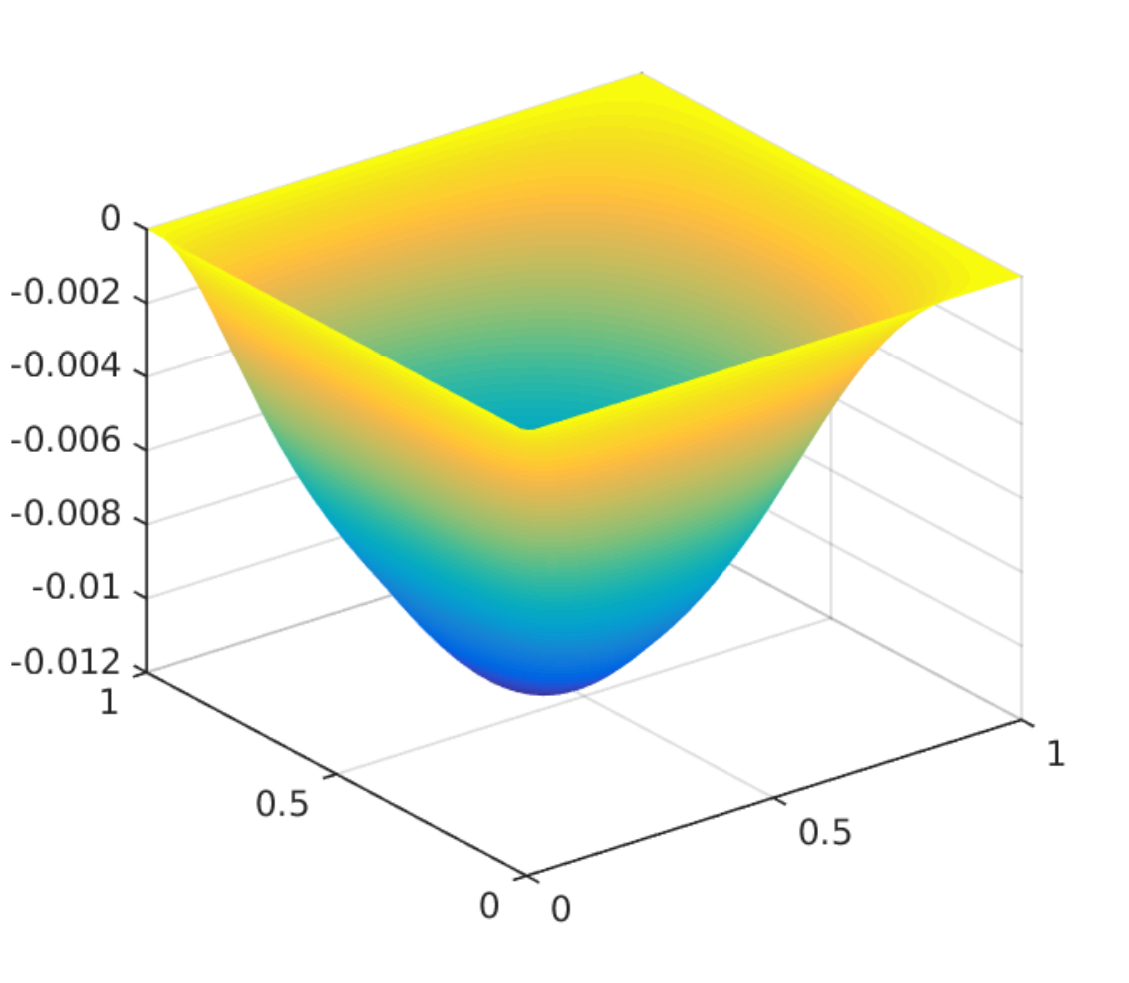}
  \\
  \includegraphics[width=7cm]{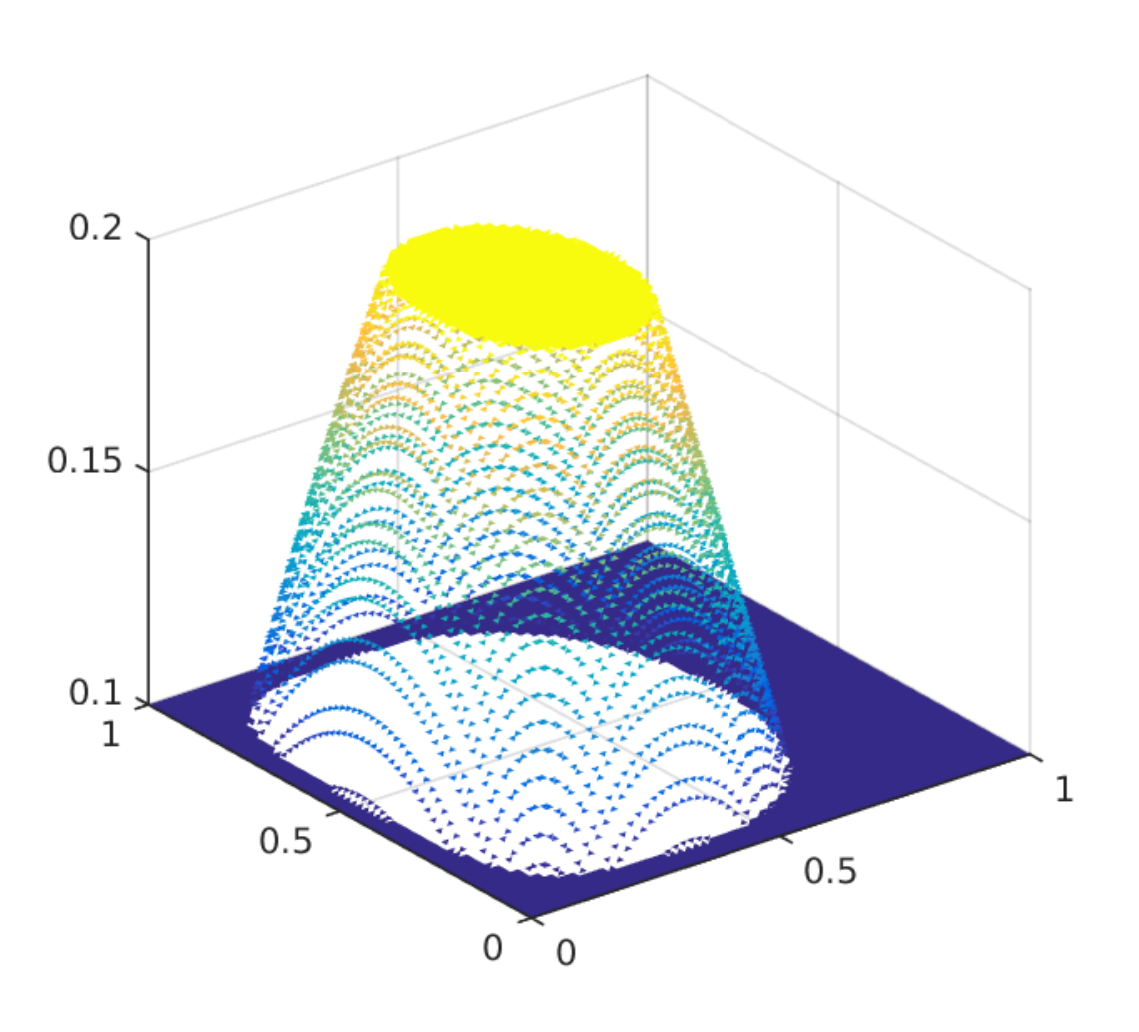}
  \includegraphics[width=7cm]{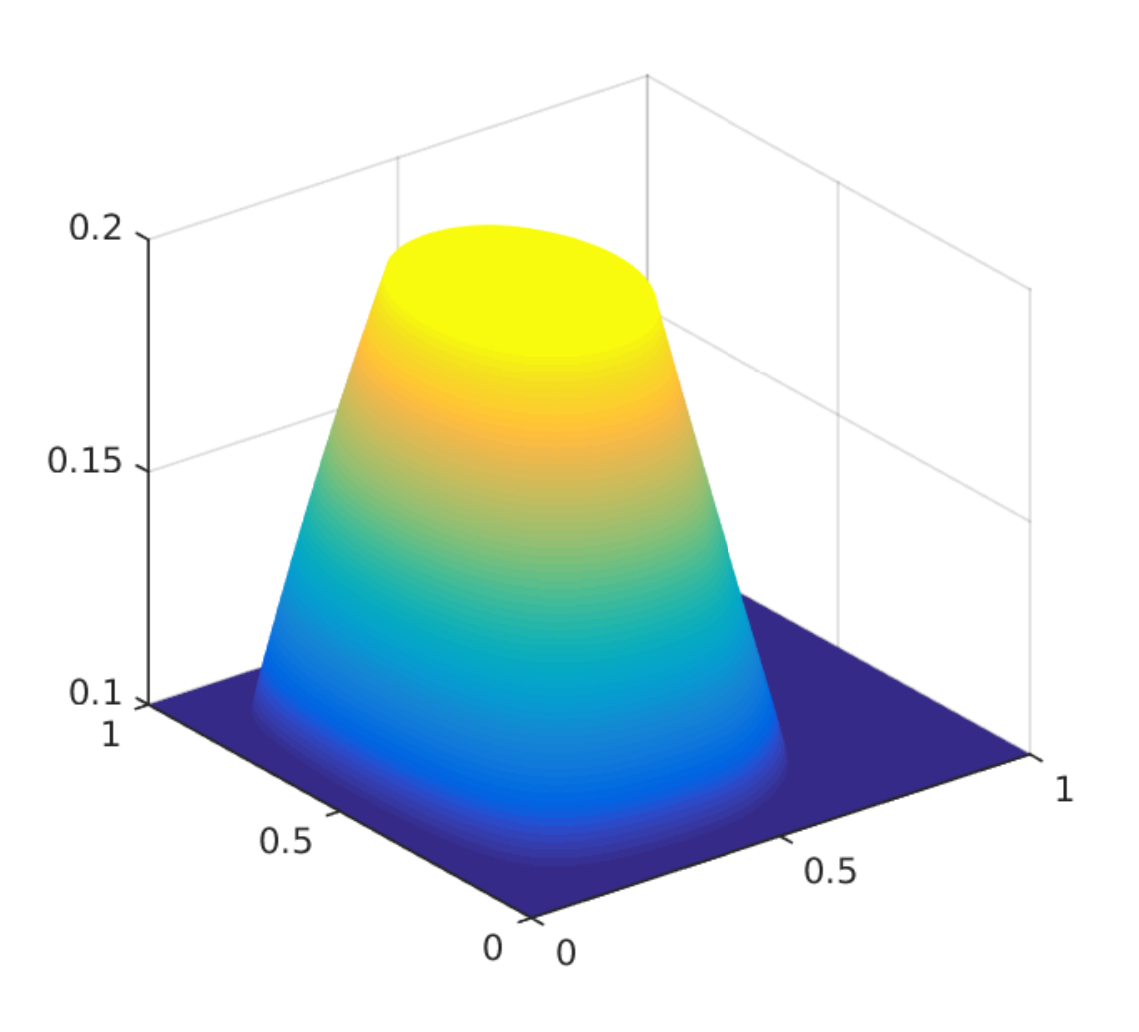}
  \caption{Final states~$y_{h}^{\gamma}$ (top row)
   and corresponding controls~$u_{h}^{\gamma}$ (bottom row) for~$h = h_7$
   in the path-following runs of Example~2. Piecewise constant (left) and
   piecewise linear and continuous (right) state
   approximation for~$\gamma = \text{2.5600e+4}$ and~$\gamma = \text{6.5536e+4}$, resp.}
  \label{fig:test3_1}
 \end{figure}
\fi
\ifpdf  
 \begin{figure}[tp] 
  \centering
  \includegraphics[width=7cm]{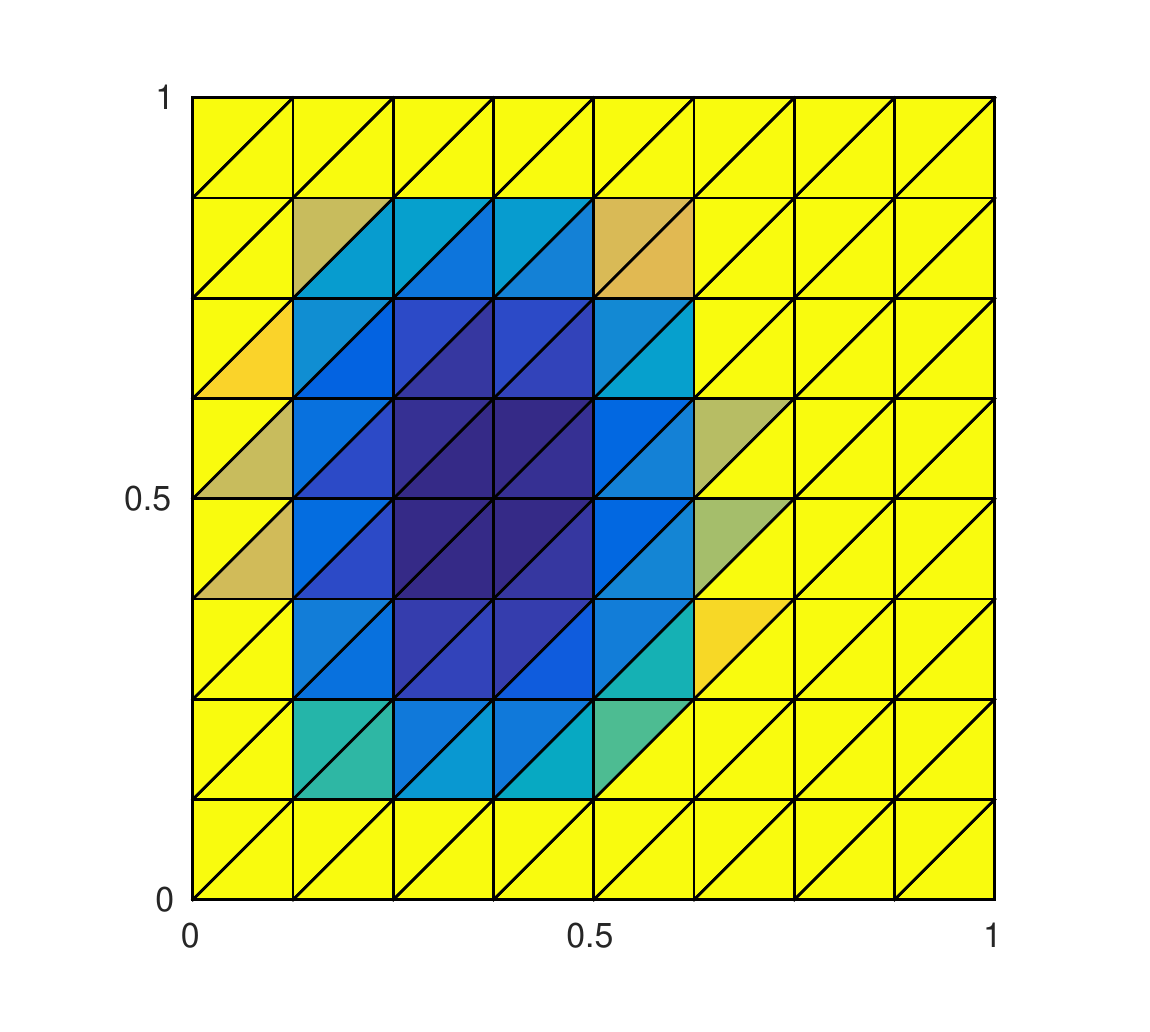}
  \includegraphics[width=7cm]{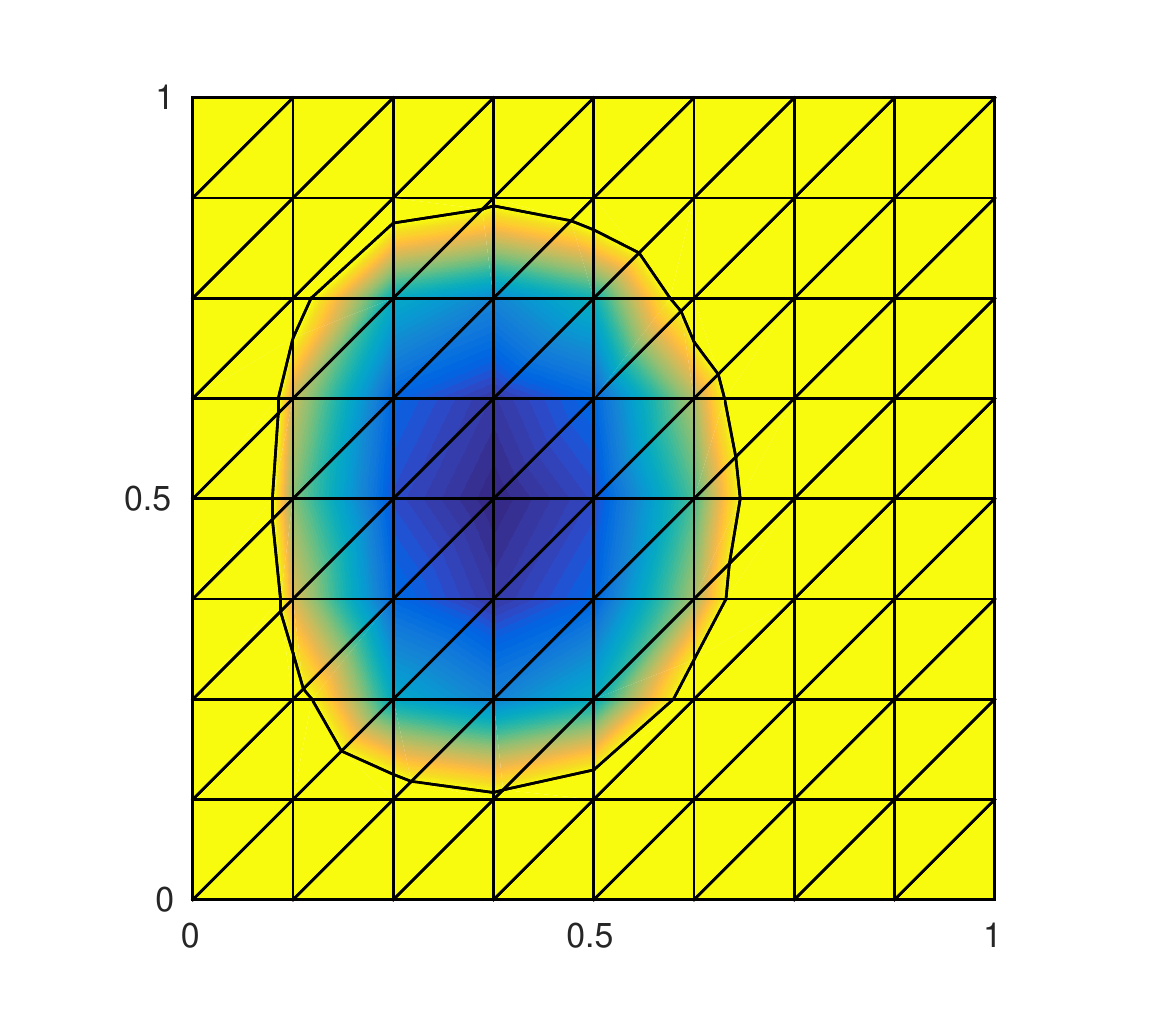}
  \caption{Comparison of the upper active set boundaries of the control~$l$
  in Example~2 
   with~$h = h_5$ and~$\gamma = 500$ in the Raviart-Thomas (left) and the
   piecewise linear, continuous case (right).}
  \label{fig:test_ctrl_activesets_bdry}
 \end{figure}
\fi

Figure~\ref{fig:test3_1} displays the resulting approximations to the optimal
state~$y$ and control~$u$ on the grid with~$h = h_7$. Similar to the
experiment in~\cite[Figure~4-1]{ArnautuLangmachSprekelsTiba2000} with
changing sign of the load,
we observe a bang-bang-like control with minimal thickness along the
boundary.

Let us comment on the active set shapes of the variationally discretized
control variable~$l$, compare the discussion in
Subsection~\ref{subsec:pwconstant_vs_pwlinear}. Since the approximation of the
state with $RT_0$-elements yields piecewise constant control
approximations the boundary of the active set in this case follows the
finite element mesh.
For piecewise linear, continuous states the control active set is generally
bounded by piecewise, non-degenerate hyperbolas.
In Figure~\ref{fig:test_ctrl_activesets_bdry} this is
depicted for
the numerical solutions of Example~2 with~$\gamma = 500$ on the mesh
with~$h = h_5$. As expected, in the piecewise linear,
continuous case the boundary of
the upper active set is already well approximated on this coarse mesh,
given the rather small
regularization parameter.


\setcounter{equation}{0}
\renewcommand{\theequation}{A.\arabic{equation}}

\appendix
\section*{Appendix A:\ \ \ Semismooth Newton method} \label{app_a:newton}
In order to solve the subproblems~$\discreteProblemMY$
and~$\discretePwLinearProblemMY$ via a semismooth Newton
method~(cf., e.g., \cite{HintermuellerItoKunisch2003}), we rewrite the
system of primal and adjoint
equations of the associated optimality 
systems, \eqref{os_my_h:prim_eq}--\eqref{os_my_h:adj_eq}
and~\eqref{os_my_h:prim_eq_p1}--\eqref{os_my_h:adj_eq_p1}, respectively,
in the matrix form
\begin{align} \label{eq:Dgammah_matrix_form}
        F^{\gamma}(\mathbf{x^{\gamma}}) \defined        \left(
        \begin{array}{cc}
         A     & k_{1}\\
         k_{2}^{\gamma} & A \\
        \end{array}
        \right)
        (\mathbf{x^{\gamma}})
        = 0,
\end{align}
where the variational inequality, in the equivalent form of the projection
formula, has been substituted for the control in the primal equation. The matrix~$A$
denotes the according finite element system matrix of the respective
discretization ansatz, whereas the
operators~$k_1$ and~$k_2^{\gamma}$ result from the right-hand sides of the primal and
adjoint equations, respectively. Due to the underlying projections these
operators will not be Fr\'echet-differentiable.
Similarly to~\cite{GuentherTber2009} we define the following generalized
derivative of~$F^{\gamma}$ as
\begin{align*}
 DF^{\gamma}(\mathbf{x}) \defined
 \left(
 \begin{array}{cccc}
  A                  & Dk_{1}(\mathbf{x})\\
  Dk_{2}^{\gamma}(\mathbf{x}) & A \\
 \end{array}
 \right),
\end{align*}
with~$D$ denoting the generalized derivative of Clarke~(cf.~\cite{Clarke1983}).

In the Raviart-Thomas case we use from~\cite{BahriawatiCarstensen2005} the
Matlab code~\emph{EBmfem} in order to discretize the Poisson
equations and refer the reader to this reference for further details. With~$m$ and~$n$
denoting the numbers of edges and elements of the triangulation,
let then~$\mathbf{x^{\gamma}} = (\mathbf{v}_{h}^{\gamma}, y_{h}^{\gamma},
\mathbf{v}_{q,h}^{\gamma}, q_{h}^{\gamma})^{T} \in \dsR^{2mn}$ be the combined
primal and dual state vector, and~$\mathbf{x} = (\mathbf{v}_{h}, y_{h},
\mathbf{v}_{q,h}, q_{h})^{T} \in \dsR^{2mn}$ another vector. We have
\begin{align*}
 k_1(\mathbf{v}_{q,h}^{\gamma}, q_{h}^{\gamma}) &\defined \left( k_1(q_{h}^{\gamma},T_{j})
  \right)_{j = 1,\ldots,n}, \\
 k_2^{\gamma}(\mathbf{v}_{h}^{\gamma}, y_{h}^{\gamma}) &\defined
  \left( k_2^{\gamma}(y_{h}^{\gamma},T_{j})
  \right)_{j = 1,\ldots,n},
\end{align*}
where
\begin{align}
 k_1(q_{h}^{\gamma},T) &\defined
 \begin{cases}  \label{eq:newton_k}
  z_{h|T} (3\,q_{h|T}^{\gamma}\,z_{h|T})^{-3/4}\abs{T}, & T \in i(q_{h}^{\gamma}),\\
  M^{-3} \abs{T} z_{h|T},                               & T \in l(q_{h}^{\gamma}),\\
  m^{-3} \abs{T} z_{h|T},                               & T \in u(q_{h}^{\gamma}),
 \end{cases}\\
 k_2^{\gamma}(y_{h}^{\gamma},T) &\defined
 \begin{cases} \nonumber
  \gamma ( y_{h|T}^{\gamma} + \tau) \abs{T}, & T \in a(y_{h}^{\gamma}),\\
  0,                                         & \text{else},
 \end{cases}
\end{align}
with the control and state active and inactive sets
\begin{align*}
 i(q_{h}^{\gamma}) &\defined \Set{T \in \mathcal{T}_{h} |
  m^{4} < 3 q_{h}^{\gamma} z_{h} < M^4},\\
 u(q_{h}^{\gamma}) &\defined \Set{T \in \mathcal{T}_{h} |
  3 q_{h}^{\gamma} z_{h} \leq m^4},\\
 l(q_{h}^{\gamma}) &\defined \Set{T \in \mathcal{T}_{h} |
  M^{4} \leq 3 q_{h}^{\gamma} z_{h}},    
\end{align*}
and
\begin{align*}
 a(y_{h}^{\gamma}) \defined \Set{T \in \mathcal{T}_{h} |
  y_{h}^{\gamma} + \tau \leq 0}.
\end{align*}
The generalized derivatives of~$k_1$ and~$k_2^{\gamma}$ are given by the diagonal matrices
\begin{align*}
 Dk_1(\mathbf{x}) &\defined \operatorname{diag}
  \left( Dk_1(q_{h},T_{j}) \right)_{j = 1,\ldots,n},\\
 Dk_2^{\gamma}(\mathbf{x}) &\defined \operatorname{diag}
  \left( Dk_2^{\gamma}(y_{h},T_{j}) \right)_{j = 1,\ldots,n},
\end{align*}
where
\begin{align}
 Dk_1(q_{h},T) &\defined
 \begin{cases} \label{eq:newton_Dk}
  -\frac{9}{4} z_{h|T}^2 \abs{T} (3\,q_{h|T}\,z_{h|T})^{-7/4}, &
   T \in i(q_{h}),\\
  0,                                                           & \text{else},
 \end{cases} \\
 Dk_2^{\gamma}(y_{h},T) &\defined
 \begin{cases} \nonumber
  \gamma \abs{T}, & T \in a(y_{h}),\\
  0,              & \text{else}.
 \end{cases}
\end{align}

\begin{algorithm}[t]
 \caption{~Sketch of the path-following method for solving problem~$\dualProblem$.}
 \label{alg:path-following}
 \begin{algorithmic}[1] 
  \STATE $h_0,\gamma_0 \gets$ positive initial grid size and regularization parameter
  \STATE $\mathbf{x}_{0} \gets$ given initial state vector of appropriate dimension,
         with subvectors~$y_0, q_0$
  \STATE $n \gets 0$
  \LOOP
   \STATE $z_{n} \gets$ (scalar component of) $G_{h_{n}}(f)$
   \WHILE{$\norm[\big]{F^{\gamma_{n}}(\mathbf{x}_{n})} > \text{given tolerance}$}
    \STATE $\mathbf{x}_{n} \gets \mathbf{x}_{n} - DF^{\gamma_{n}}(\mathbf{x}_{n})^{-1}
           F^{\gamma_{n}}(\mathbf{x}_{n})$
   \ENDWHILE
   \STATE $l_{n} \gets \left( P_{\left[m^{4},M^{4}\right]} \left( 3\,q_{n}\,z_{n}
          \right) \right)^{-3/4}$
   \IF{$J^{\gamma_{n}}(y_{n},l_{n}) =
         \norm[\big]{l_{n}^{-1/3}}_{L^{1}(\Omega)} + \gamma_{n}/2
         \norm[\big]{(y_{n}+\tau)^{-}}_{L^{2}(\Omega)}^{2}
         \text{ is sufficiently small}$}
    \STATE Stop and return the last iterate
   \ELSE
    \STATE $h_{n+1} \gets h_{n} / 2$
    \STATE $\gamma_{n+1} \gets \gamma_{n} \cdot 2^{\kappa} \quad
           \text{(with~$\kappa$ chosen as suggested by
            Theorems~\ref{thm:error_overall} and~\ref{thm:error_overall_p1})}$
    \STATE $\mathbf{x}_{n+1} \gets$ interpolation of~$\mathbf{x}_{n}$ on the
           refined mesh
    \STATE $n \gets n+1$
   \ENDIF
  \ENDLOOP
 \end{algorithmic}
\end{algorithm}
In the case of piecewise linear elements, let~$N$ be the number of inner nodes
in the triangulation and~$(\phi_j)_{j=1}^N$ a canonical basis of~$Y_h$. 
With~$\mathbf{x}^\gamma = (y^\gamma_h, q^\gamma_h)^T \in \dsR^{2N}$ and an
arbitrary~$\mathbf{x} = (y_h, q_h)^T \in \dsR^{2N}$ we have that
\begin{align}
 \begin{split} \label{eq:newton_k_p1}
 k_1(q^{\gamma}_h) &\defined \left( -\int_\Omega z_h
  \left(P_{[m^4,M^4]}(3\,q^\gamma_h\,z_h) \right)^{-3/4} \phi_j \ud x \right)_{j=1,\ldots,N},\\
 k_2^{\gamma}(y^{\gamma}_h) &\defined \left( -\int_\Omega \gamma\, P_{(-\infty,0]}(y^\gamma_h+\tau) 
  \, \phi_j \ud x \right)_{j=1,\ldots,N}.
 \end{split}
\end{align}
Denoting by~$g_{1} \in \partial P_{[m^4,M^4]}$ and~$g_{2} \in \partial P_{(-\infty,0]}$
subgradients of the projection operators, we
further obtain
\begin{align}
 \begin{split} \label{eq:newton_Dk_p1}
 Dk_1(\mathbf{x})_{j,k} &\defined
  \int_\Omega 
  \frac{9}{4} z_h^2
  \left( P_{[m^4,M^4]}(3\,q_h\,z_h) \right)^{-7/4}
  g_{1}(3\,q_h\,z_h) \, \phi_k \, \phi_j \ud x,\\
 Dk_2^{\gamma}(\mathbf{x})_{j,k} &\defined -\int_\Omega
  \gamma\, g_{2}(y_h+\tau) \, \phi_k \, \phi_j \ud x.
 \end{split}
\end{align}

The nonlinear equation~\eqref{eq:Dgammah_matrix_form} can now be solved with
a semismooth Newton method. For this purpose, and an initial
iterate~$\mathbf{x}_{0}$, we generate a sequence of semismooth Newton steps via
\begin{align*}
 \mathbf{x}_{n+1} \defined \mathbf{x}_{n} - DF^{\gamma}(\mathbf{x}_{n})^{-1}
  F^{\gamma}(\mathbf{x}_{n}),
 \qquad n\in \dsN,
\end{align*}
which, very similar to Proposition~4.1 of~\cite{GuentherTber2009}, can be
shown to be well-defined and locally convergent.

In order to approximate the solution of~$\dualProblem$ one can perform
a path-following method, as outlined in Algorithm~\ref{alg:path-following}.
To this end, the above iteration can be continued until
the Euclidean norm of the vector~$F^{\gamma}(\mathbf{x}_{n})$, which
measures how much~$\mathbf{x}_{n}$ violates the optimality system, is below a
given tolerance, and use the obtained
approximate solution to~$\regProblemMY$ as the starting iterate on a refined mesh
with increased~$\gamma$.
This nested iteration approach helps to stay within the convergence radius of
the Moreau-Yosida-penalized Newton method. Overall, in our computations we
observed a marked sensitivity of the Newton-type method with respect to the
regularization parameter.
Should the Newton solver diverge, however, we were able to return into the
domain of convergence by choosing a more conservative increase in~$\gamma$.
As mentioned before, the associated controls can then be extracted from the state
variables~$\mathbf{x}_{n}$ by means of the projection formula~\eqref{os_my_h:proj}.


\end{document}